\documentclass[letterpaper,12pt,reqno]{amsart}
\usepackage{graphicx,amsmath,amssymb,amsthm}
\usepackage[left=1in,right=1in, top=1in, bottom=1in]{geometry}
\newtheorem{thm}{Theorem}[section]

\newtheorem{prop}[thm]{Proposition}
\newtheorem{thrm}[thm]{Theorem}
\newtheorem{lem}[thm]{Lemma}

\newtheorem{rmk}[thm]{Remark}
\newtheorem{conj}[thm]{Conjecture}

\newcommand{\QQ}{\mathbb{Q}}

\newcommand{\NN}{\mathbb{N}}
\newcommand{\ZZ}{\mathbb{Z}}
\newcommand{\FF}{\mathbb{F}}
\newcommand{\C}{{\mathcal{C}}}

\newcommand{\Z}{\mathbb{Z}}

\newcommand{\cp}{C_p^{(\ell)}(a,f,n)}
\newcommand{\cpf}{C_p^{(\ell)}(1,f,1)}

\newcommand{\sumv}{\underset{\substack{f_1,f_2\leq V }}{{\sum }'}}
\newcommand{\sumfs}{\underset{\substack{f_1,f_2\geq 1 }}{{\sum }'}}
\makeatletter
\def\imod#1{\allowbreak\mkern7mu({\operator@font mod}\,\,#1)}
\makeatother

\numberwithin{equation}{section}

\title[An asymptotic for the average number of amicable pairs for elliptic curves]{An asymptotic for the average number of amicable pairs for elliptic curves}
\author{James Parks \\ \\ \textbf{\lowercase{with an appendix by} S\lowercase{umit} G\lowercase{iri}}}
\address{Department of Mathematics and Computer Science, University of Lethbridge,\newline
	\rule[0ex]{0ex}{0ex}\hspace{8pt}  4401 University Drive, Lethbridge, AB, T1K 3M4, Canada\newline
	\rule[0ex]{0ex}{0ex}\hspace{8pt} \textit{Present address}: Institut f\"ur Algebra, Zahlentheorie und Diskrete Mathematik, \newline
	\rule[0ex]{0ex}{0ex}\hspace{8pt} Leibniz Universit\"at Hannover, Welfengarten 1, 30167 Hannover, Germany}
\email{parks@math.uni-hannover.de}
\thanks{This work was supported by a Pacific Institute for the Mathematical Sciences Postdoctoral Fellowship.}

\date{\today}
\begin{document}

\begin{abstract}
Amicable pairs for a fixed elliptic curve defined over $\QQ$ were first considered by Silverman and Stange where they conjectured an order of magnitude for the function that counts such amicable pairs. This was later refined by Jones to give a precise asymptotic constant. The author previously proved an upper bound for the average number of amicable pairs over the family of all elliptic curves. In this paper we improve this result to an asymptotic for the average number of amicable pairs for a family of elliptic curves.
\end{abstract}

\maketitle

\section{Introduction}
Let $E$ be an elliptic curve defined over $\QQ$. For a prime $p$, let $a_p(E)$ denote the trace of the Frobenius automorphism. 
Silverman and Stange \cite{JSKS:1} define a pair $(p,q)$ of prime numbers with $p<q$ to be an \textit{amicable pair} of $E$ 
if $E$ has good reduction at both $p$ and $q$ and 
\begin{equation}
\#{E}_{p}(\FF_{p}):=p+1-a_p(E_{p})=q \quad {\rm and}\quad \#{E}_{q}(\FF_q)=p.\label{amicon}
\end{equation}
As observed in \cite[Remark 1.5]{JSKS:1} amicable pairs arose naturally when Silverman and Stange
generalized Smyth's \cite{CS:1} results on index divisibility of Lucas sequences  to elliptic divisibility sequences.

We are interested in the distribution of amicable pairs for a fixed elliptic curve $E/\QQ$. We first define the amicable pair counting function, originally considered by Silverman and Stange \cite{JSKS:1},
\begin{equation*}
\pi_{E,2}(X):=\#\{p\leq X : \#E_p(\FF_p)=q \text{ is prime and }\#E_q(\FF_q)=p\}.
\end{equation*}
They used a heuristic argument to give the following conjecture for the behavior of $\pi_{E,2}(X)$.
\begin{conj}[\textbf{Silverman-Stange}] \label{silverstangeamconj}
Let $E/\QQ$ be an elliptic curve. Assume that there are infinitely many primes $p$ such that 
$\#{E}_{p}(\FF_{p})$ is prime. Then as $X\rightarrow\infty$ we have that
\begin{align*}
&\pi_{E,2}(X) \asymp  \frac{\sqrt{X}}{(\log X)^2} \quad \text{ if }E\text{ does not have complex multiplication (CM)}\\
\text{and}\qquad&\\
&\pi_{E,2}(X)\sim  A_E\frac{X}{(\log X)^2} \quad \text{ if }E\text{ has CM},
\end{align*}
where the implied constants in $\asymp$ are both positive and depend only on $E$ and $A_E$ is a precise positive constant.
\end{conj}

\begin{rmk}
{\em (i) Silverman and Stange \cite{JSKS:1} also defined an $L$-tuple $(p_1,\ldots,p_L)$ of distinct prime numbers to be an \textit{aliquot cycle} of length $L\geq 2$ of $E$ 
if $E$ has good reduction at each prime $p_i$ and  $$\#{E}_{p_i}(\FF_{p_i})=p_{i+1}  \text{ for } 1\leq i\leq L-1 \text{ and } \#{E}_{p_L}(\FF_{p_L})=p_1.$$ They also introduced the analogously defined aliquot cycle counting function $\pi_{E,L}(X)$, and gave a conjecture for its behavior in the non-CM case. The main focus of their work was on aliquot cycles in the CM-case, where they proved that there are only finitely many under certain conditions when $L\geq 3$ as well as showing that more structure occurs in the case $L=2$.}
\end{rmk}

Jones \cite{NJ:1} refined Conjecture \ref{silverstangeamconj} in the non-CM case for aliquot cycles by using a heuristic argument similar to that of Lang and Trotter \cite{SLHT:1}. We state the refined conjecture in the particular case of amicable pairs below.
 
\begin{conj}[\textbf{Jones}]\label{jonesconj}
Let $E/\QQ$ be an elliptic curve without complex multiplication. Then there is a non-negative real constant $C_{E,2}\geq 0$ such that, as $X\rightarrow\infty$, we have that
\begin{equation*}
\pi_{E,2}(X) \sim C_{E,2}\int_2^X\frac{1}{2\sqrt{t}(\log t)^2}dt.
\end{equation*}
\end{conj}

We note that the conjectured asymptotic is consistent with Conjecture \ref{silverstangeamconj}. Moreover, Jones gave an explicit expression for the constant $C_{E,2}$ in terms of invariants of the elliptic curve $E$. We discuss this constant in greater detail below.

For a non-zero integer $n$, we denote the $n$-torsion subgroup of $E$ by $E[n]$. Let 
$\QQ(E[n])$ be the field generated by adjoining to $\QQ$ the 
$x$ and $y$-coordinates of the $n$-torsion points of $E$. We have that $E[n]\cong \ZZ/n\ZZ \times \ZZ/n\ZZ$ for $n\geq 2$.
Since each element of the Galois group  ${\rm Gal}(\bar{\QQ}/\QQ)$ acts on $E[n]$ we have that ${\rm Gal}(\QQ(E[n])/\QQ)\subseteq {\rm GL}_2(\ZZ/n\ZZ)$ (see \cite[Chapter III.7]{JS:1}).

If $[{\rm GL}_2(\ZZ/n\ZZ):{\rm Gal}(\QQ(E[n])/\QQ)]\leq 2$ for each $n\geq 1$ then $E$ is called a \textit{Serre} curve (see \cite[pp. 309-311]{JPS:1} and \cite[p. 51]{SLHT:1}). 
Jones \cite{NJ:1} has shown that for any Serre curve $E$ the conjectural constant $C_{E,2}$ is positive and $C_{E,2}=C_2\cdot f_2(\Delta_{sf}(E))$, where $\Delta_{sf}(E)$ 
denotes the square-free part of the discriminant of any Weierstrass model of $E$ and $f_2$ is a positive function which approaches 1 
as $\Delta_{sf}(E)\rightarrow \infty$. In particular, Jones \cite{NJ:1} gave the formula 
\begin{equation}
\label{eqn: matrix counts}C_2=\frac{8}{3\pi^2}\lim_{k\rightarrow \infty}\frac{n_k^2\cdot\#\Big\{(g_1,g_2) \in \text{GL}_2(\ZZ/n_k\ZZ)^2 : \begin{array}{l}\det(g_2)\equiv \det(g_1)+1-\text{tr}(g_1)\imod{n_k} \\ \det(g_1)\equiv \det(g_2)+1-\text{tr}(g_2)\imod{n_k}\end{array}\Big\}}{|\text{GL}_2(\ZZ/n_k\ZZ)|^2},
\end{equation}
where $n_k=\prod_{\ell \leq k}\ell^k.$

There are currently no techniques known to approach conjectures like Conjecture \ref{jonesconj} for a single elliptic curve. A much more tractable problem is to consider the behaviour of $\pi_{E,2}(X)$ averaged over a family of elliptic curves. This approach has been used successfully to address many other problems related to distributions of invariants of elliptic curves. The most well known is the Lang-Trotter Conjecture \cite{SLHT:1} which counts the number of primes $p\leq X$ such that $a_p(E)=t$ for a fixed integer $t$. The Lang-Trotter Conjecture was first shown to hold on average over a family of elliptic curves in the case $t=0$ by Fouvry and Murty \cite[Thoerem 6]{EFMM:1} and then extended to the case of nonzero integers by David and Pappalardi \cite{CDFP:1}.

We let $a$ and $b$ be integers and let $E_{a,b}$ be the elliptic curve given by the Weierstrass equation 
$$E_{a,b}: y^2=x^3+ax+b,$$ with discriminant $\Delta(E_{a,b})\neq 0$. For $A,B>0$ we consider the two parameter family of elliptic curves as 
\begin{equation} \label{ecfamily}
\C:=\C(A,B)=\{E_{a,b}: |a|\leq A, |b|\leq B, \Delta(E_{a,b})\neq 0\}.
\end{equation}

Conjecture \ref{jonesconj} was considered on average in \cite{JP:1} over the family $\C(A,B)$ in $\eqref{ecfamily}$ where the conjectured upper bound for the average number of aliquot cycles with small bounds on the size of $A$ and $B$ was obtained (cf. \cite[Theorem 1.6]{JP:1}). In this paper we extend this result to an asymptotic on average in the particular case of amicable pairs. We first state this result given in terms of a sum over primes.

\begin{thrm} \label{maintheorem}
	Let $\epsilon >0$, let $E/\QQ$ be an elliptic curve, and let $\C$ be the family of elliptic curves in $\eqref{ecfamily}$ with
	$$A,B>X^{\epsilon} \quad {\rm and} \quad X^3(\log X)^6< AB <e^{X^{\frac{1}{6}-\epsilon}}.$$ Then we have that
	\begin{equation*}  
	\frac{1}{|\C|}\sum_{E\in \C}\pi_{E,2}(X) = \frac{8}{3\pi^2}\sum_{p\leq X}\frac{C_2(p)}{\sqrt{p}\log p} + O\left(\frac{\sqrt{X}}{(\log X)^{2+\epsilon}}\right),
	\end{equation*}
	where
	\begin{align}
	C_2(p):=&\frac{4}{9}\prod_{\ell>2}\Bigg(1-\frac{(2\ell^4+3\ell^3)\left(\frac{p-1}{\ell}\right)^2+\ell^3\left(\frac{p}{\ell}\right)-\ell^4+2\ell^3+4\ell^2-1}{(\ell^2-1)^3}\Bigg).\label{actualkzerop}
	\end{align}
\end{thrm}

In an appendix by Sumit Giri it will be shown that if we average the function $C_2(p)$ defined in $\eqref{actualkzerop}$ over the set of primes up to $X$ then we will obtain the constant $C$ given in $\eqref{sumitconstant}$. Applying Theorem \ref{thrm: Giri} and partial summation we obtain the following result.

\begin{thrm} \label{maintheoremsumit}
Under the same conditions as Theorem \ref{maintheorem} we have that 
	\begin{equation}\label{eqn: avg asymp}  
	\frac{1}{|\C|}\sum_{E\in \C}\pi_{E,2}(X) = \frac{8C}{3\pi^2}\frac{\sqrt{X}}{(\log X)^2} + O\left(\frac{\sqrt{X}}{(\log X)^{2+\epsilon}}\right),
	\end{equation}
	where 	\begin{align}
	C:=&\prod_{\ell}\left(1-\frac{(2\ell^4+3\ell^3)(\ell-2)-(\ell-1)(\ell^4-2\ell^3-4\ell^2+1)}{(\ell-1)(\ell^2-1)^3}\right)\label{sumitconstant}.
	\end{align}
\end{thrm}

\begin{rmk}
	\emph{The average number of aliquot cycles, and in particular amicable pairs, has been independently studied by David, Koukoulopolous and Smith \cite[Theorem 1.6]{DKS:1} using different techniques building upon a theorem of Gekeler \cite[Theorem 5.5]{EG:1}. They also obtain an asymptotic result on average with an average constant expressed as }

		\begin{equation*}
			\label{eqn: matrix counts gek}C'=\frac{8}{3\pi^2}\prod_{\ell}\lim_{k\rightarrow \infty}M(\ell,k)
		\end{equation*}
		\emph{where} $$M(\ell,k):=\frac{\ell^{2k}\cdot\#\Big\{(\sigma_1,\sigma_2) \in \text{GL}_2(\ZZ/\ell^k\ZZ)^2 : \begin{array}{l}\det(\sigma_2)\equiv \det(\sigma_1)+1-\emph{\text{tr}}(\sigma_1)\imod{\ell^k} \\ \det(\sigma_1)\equiv \det(\sigma_2)+1-\emph{\text{tr}}(\sigma_2)\imod{\ell^k}\end{array}\Big\}}{|\text{GL}_2(\ZZ/\ell^k\ZZ)|^2}.$$
		\emph{However, we note that the constant $C$ is not an obvious consequence of the limit  of $M(\ell,k)$.}
\end{rmk}

\begin{rmk}
	{\em Jones \cite{NJ:1} also determined the Euler product representation of $M(\ell,1)$ which can be used to find the first approximation of $C_2$. However, it is not always the case that $M(\ell,1)=M(\ell,k)$ for $k>1$\footnote{For example, $M(3,1)\neq M(3,2)$.}. In fact, for $\ell>2$ it is unclear if $M(\ell,k)$ stabilizes at all.}
\end{rmk}

The main result of this paper, Theorem \ref{maintheorem}, significantly improves \cite[Theorem 3.1]{JP:1} in the case of amicable pairs, that is, when $L=2$ from an upper bound on average to an asymptotic result on average. The main tools used in this improvement are standard applications of Duering's theorem and the analytic class number formula along with a generalization of the approach of David and Smith \cite{CDES:1}, \cite{CDES:2}.

For an elliptic curve $E/\QQ$ and a fixed integer $N$, David and Smith \cite{CDES:1}, \cite{CDES:2} and Chandee, David, Koukoulopoulos and Smith \cite{CDKS:1} considered the related function that counts the number of primes $p$ such that $\#E_p(\FF_p)=N$. Many of the techniques used to sum class numbers in the proofs of \cite[Theorem 3 and Theorem 7]{CDES:1} as well as \cite[Proposition 5.1]{CDKS:1} generalize to the case of amicable pairs. However, the case of amicable pairs is more technical since we must now consider a sum of a product of class numbers.

\subsection{Acknowledgment} This work was first started during my PhD under my advisor, Chantal David and I would like to thank her for all her great advice and insight while working on this problem. Also, I am extremely grateful to Amir Akbary for his encouragement and helpful intuition during the writing of this paper. I would also like to thank Sumit Giri for his careful reading and suggestions. Finally, I would like to thank Nathan Jones and Dimitris Koukoulopoulos for their very helpful discussions related to this paper.

\section{Preliminaries}
We first fix notation. Throughout this paper we use $\ell, p,$ and $q$ to denote primes. For an elliptic curve $E/\QQ$, we define the lower and upper limits of the Hasse bound as
\begin{equation}
p^-:=p+1-2\sqrt{p}< \# E_{p}(\FF_{p}) < p^+:=p+1+2\sqrt{p}.
\end{equation}
Let $n$ be a positive integer $n$. Then we let $P^+(n)$ denote its largest prime factor and define $\nu_\ell(n)$ to be the non-negative integer $\alpha$ such that $\ell^\alpha \parallel n$. Let $m$ be a positive integer, then we define $\kappa_m(n)$ to be the multiplicative function defined on prime powers by
\begin{equation}\label{kappadef}
\kappa_{m}(\ell^{\nu_\ell(n)}):=\begin{cases} \ell & \text{ if } 2\nmid \nu_\ell(n) \text{ and } \ell \nmid m, \\ 1 &\text{ otherwise}.\end{cases}
\end{equation}
Let $\chi_{d}$ be the real Dirichlet character given by the Kronecker symbol 
\begin{align*}
\chi_{d}(n):=\left( \frac{d}{n}\right)
\end{align*}
and associated Dirichlet series given by $$L(s,\chi_d):=\sum_{n=1}^\infty \frac{\chi_d(n)}{n^s}= \prod_{\ell}\left(1-\frac{\chi_d(\ell)}{\ell^s}\right)^{-1} \quad \text{ for }\Re (s)>1.
$$ Then for $y>1$ we define the truncated quadratic Dirichlet $L$-function as $$L(1,\chi_d;y):=\prod_{\ell\leq y}\left(1-\frac{\chi_d(\ell)}{\ell}\right)^{-1}.$$ 
We also make use of the notation
\begin{equation}
E(X,Y;q):=\max_{(a,q)=1}\Bigg|\sum_{\substack{X< p \leq X+Y \\ p\equiv a \imod{q}}}\log p -\frac{Y}{\varphi(q)}\Bigg|.\label{eerror}
\end{equation}
Lastly, for positive integers $m$ and $n$ we consider the symmetric function
\begin{equation}\label{dmn}
D(m,n):=(m+1-n)^2-4m=D(n,m),
\end{equation}
which occurs frequently in our calculations.

We require the following two technical results in the proofs of Theorem \ref{maintheorem} and Proposition \ref{propdsseven}. The first proposition is a consequence of a result of Granville and Soundararajan \cite{AGKS:1} which is essentially due to Elliot \cite{PE:1}. It allows us to bound the error terms in our calculations in Proposition \ref{propdsseven}.

\begin{prop}[\textbf{Granville-Soundararajan}] \label{lemtwothree}
Let $\alpha\geq 1$ and $Q\geq 3$. There is a set ${\mathcal{E}}_\alpha(Q)\subset[1,Q]$ of at most 
$Q^{\frac{2}{\alpha}}$ integers such that if $\chi$ is a Dirichlet character modulo $q\leq \exp\{(\log Q)^2\}$, whose conductor does not belong to ${\mathcal{E}}_\alpha(Q)$, then 
$$L(1,\chi)=L(1,\chi;(\log Q)^{8\alpha^2})\left(1+O_\alpha\left(\frac{1}{(\log Q)^\alpha}\right)\right).$$
\end{prop}

\begin{proof}
The result is stated in terms of primitive characters in \cite[Proposition 2.2]{AGKS:1}. The proof of the proposition in its present form is given in \cite[Lemma 2.3]{CDKS:1}. 
\end{proof}

The second result we require is the following version of the Bombieri-Vinogradov theorem for primes in short arithmetic progressions. 

\begin{lem}[\textbf{Koukoulopoulos}] \label{lemthreefour}
Let $\epsilon>0$ and let $A\geq 1$. For $2\leq Y\leq X$ and $1\leq Q^2\leq Y/X^{1/6+\epsilon}$, we have that $$\int_X^{2X}\sum_{q\leq Q}E(u,Y;q)du\ll \frac{XY}{(\log X)^A}.$$
\end{lem}

\begin{proof}
This result follows from \cite[Theorem 1.1]{DK:1}.
\end{proof}

We now state the analytic class number formula for quadratic Dirichlet $L$-functions (cf. \cite[Chapter 6]{HD:1}).

\begin{thrm} \label{analclass}
Let $D= df^2$ be a negative number such that $d$ is a negative fundamental discriminant and let $\chi_d$ be the Kronecker symbol. Then 
$$\frac{h(d)}{w(d)}= \frac{\sqrt{-d}}{2 \pi} L(1,\chi_d),$$ where $h(d)$ denotes the usual class number of the imaginary quadratic order of discriminant $d$ and $w(d)$ is the number of roots of unity in $\QQ(\sqrt{d})$.
\end{thrm}

We recall the following formulation of the definition of the Hurwitz-Kronecker class number (cf. \cite{HL:1}). Let $D$ be a negative (not necessarily fundamental) discriminant. Then the \textit{Hurwitz-Kronecker class number} of discriminant $D$ is defined by
\begin{equation*}
H(D)=\sum_{\substack{f^2 | D \\ \frac{D}{f^2}\equiv 0,1 \imod{4}}}\frac{h\left(\frac{D}{f^2}\right)}{w\left(\frac{D}{f^2}\right)}, 
\end{equation*}
where $w(D)$ is the number of roots of unity contained in $\QQ(\sqrt{D})$. This formulation leads to the following useful result of Deuring \cite{MD:1}.

\begin{thrm}[\textbf{Deuring}] \label{deuring}
Let $p>3$ be a prime and let $t$ be an integer such that $t^2-4p<0$. Then
$$\sum_{\substack{\bar{E}/\FF_{p} \\ a_p(\bar{E})=t}} \frac{1}{\#{\rm Aut}(\bar{E})}=H(t^2-4p),$$
where $\bar{E}$ denotes a representative of an isomorphism class of $E/\FF_p$. 
\end{thrm}

The following result \cite[Theorem 3.1]{JP:1} allows us to interpret the average number of amicable pairs in terms of a sum of Hurwitz-Kronecker class numbers.
\begin{thrm}\label{aliquotaverage}
Let $\epsilon >0$, let $E/\QQ$ be an elliptic curve, and let $\C$ be the family of elliptic curves in $\eqref{ecfamily}$ with
 $$A,B>X^{\epsilon} \quad \text{ and} \quad X^3(\log X)^6< AB <e^{X^{\frac{1}{6}-\epsilon}}.$$ Then as $X\rightarrow \infty$ we have that
\begin{equation}
\frac{1}{|\C|}\sum_{E\in \C}\pi_{E,2}(X)=\Bigg\{\sum_{p\leq X}\frac{1}{p}\sum_{p^-<q<p^+}\frac{H(D(p,q))^2}{q}\Bigg\}\left(1+
O\left(\frac{1}{X^\epsilon}\right)\right).\label{avgthrmsum}
\end{equation}
\end{thrm}

In the analysis of the inner sum in $\eqref{avgthrmsum}$ we require the following two lemmas. 
\begin{lem}[\textbf{David-Smith}] \label{dslemtwelve}
Let $f,m,n$ be positive integers. Then $$\#\{m\in \ZZ/f\ZZ: D(m,n)\equiv 0 \imod{f}\}\ll \sqrt{f}.$$
\end{lem}

\begin{proof}
The proof is given in \cite[Lemma 12]{CDES:1}. 
\end{proof}

\begin{lem}[\textbf{David-Smith}] \label{lemdseight}
Let $p$ be a prime and let $x$ be a positive real number. Then $$\sum_{n>x} \frac{1}{\kappa_{2p}(n)\varphi(n)}\ll \frac{1}{\sqrt{x}} \quad \text{ and } \quad \sum_{n\geq 1} \frac{1}{\kappa_{2p}(n)\varphi(n)}\ll 1.$$
\end{lem}

\begin{proof} The result follows by specializing to the case where $N$ is a prime in \cite[Lemma 8]{CDES:1}.
\end{proof}

\section{An asymptotic result for a sum of a product of class numbers}
The goal of this section is to first determine an asymptotic result for the inner sum in $\eqref{avgthrmsum}$ and then to give the proof of the main result, Theorem \ref{maintheorem}. 

\begin{thrm} \label{thrmdsthree}
Let $\gamma>0$, let $p,q$ be distinct primes, and let $\epsilon >0$. Define $\displaystyle{Y:=\frac{\sqrt{p}}{(\log p)^{\gamma+5}}}$, then we have that
\begin{multline}  
\sum_{p^-< q< p^+}\frac{H(D(p,q))^2}{q} = \frac{8C_2(p)\sqrt{p}}{3\pi^2\log p} + O\Bigg(\frac{\sqrt{p}}{(\log p)^{1+\gamma}}\\
+\frac{1}{\log p}\sum_{\frac{-2\sqrt{p}}{Y}<k\leq \frac{2\sqrt{p}}{Y}}\sum_{f_1,f_2\leq (\log 4p)^{32+8\gamma}}f_1f_2\sum_{n_1,n_2\leq p^{\epsilon}}E(p+1+kY,Y;4[n_1f_2^2,n_2f_2^2])\Bigg),\label{asympresult}
\end{multline}
where $C_2(p)$ is given in $\eqref{actualkzerop}$.
\end{thrm}

\begin{proof}
We first consider the left hand side of $\eqref{asympresult}$ and divide the interval $(p^-,p^+)$ into intervals of length $Y$. We define these subintervals as $$I:= \left[\frac{-2\sqrt{p}}{Y},\frac{2\sqrt{p}}{Y} \right)\cap \ZZ,$$
where for each $k\in I$, we write $X_k:=p+1+kY$. Next we let $d_i:=D(p,q)/f_i^2$ for $i=1,2$ and define $$\chi_{d_1}:=\left(\frac{D(p,q)/f_1^2}{\cdot}\right) \quad \text{and} \quad
\chi_{d_2}:=\left(\frac{D(p,q)/f_2^2}{\cdot}\right).$$ 
For $p,q>2$ we have that $D(p,q)$ is odd and hence $f_i$ is also odd for $i=1,2$. Therefore the condition $D(p,q)/f_i^2\equiv 1 \imod{4}$ is always satisfied for $i=1,2.$ Applying Theorem \ref{analclass} gives
\begin{align}
\sum_{p^-< q< p^+}\frac{H(D(p,q))^2}{q}&=\sum_{k\in I}\sum_{X_k<q\leq X_k+Y}\frac{H(D(p,q))^2}{q} \nonumber \\
&=\frac{1}{4\pi^2}\sum_{k\in I}\sum_{X_k<q\leq X_k+Y}\frac{|D(p,q)|}{q}\sum_{\substack{f_1^2, f_2^2 | D(p,q)}}\frac{L(1,\chi_{d_1})L(1,\chi_{d_2})}{f_1f_2}\nonumber\\
&=\frac{1}{4\pi^2}\sum_{k\in I}\sum_{\substack{f_1,f_2\leq 2\sqrt{X_k+Y}   }}\frac{1}{f_1f_2}
\sum_{\substack{X_k<q\leq X_k+Y \\ f_1^2, f_2^2 | D(p,q)}}\frac{|D(p,q)|L(1,\chi_{d_1})L(1,\chi_{d_2})}{q}.\label{hamicablesum}
\end{align}
We now focus on the inner sum of $\eqref{hamicablesum}$ and write $$\frac{|D(p,q)|}{q}=\frac{|D(p,X_k)|\log q}{p \log p} + \left(\frac{|D(p,q)|}{q}-\frac{|D(p,X_k)|\log q}{p \log p}\right).$$ 
If $q$ is a prime in the interval $(X_k,X_k+Y]$, then $q=X_k+O(Y),$ and hence $D(p,q)=D(p,X_k)+O(Y\sqrt{p}).$ Since it is also true that 
$q=p+O(\sqrt{p})$ we have that $$\left| \frac{|D(p,q)|}{q}-\frac{|D(p,X_k)|\log q}{p\log p}\right|\ll \frac{|D(p,X_k)|}{p^\frac{3}{2}}
+\frac{Y}{\sqrt{p}}.$$ 
Thus $\eqref{hamicablesum}$ becomes
\begin{align}
&\sum_{p^-< q< p^+}\frac{H(D(p,q))^2}{q}\nonumber\\
=&\frac{1}{4\pi^2p\log p}\sum_{k\in I}|D(p,X_k)|\sum_{\substack{f_1,f_2\leq 2\sqrt{X_k+Y}   }}\frac{1}{f_1f_2}
\sum_{\substack{X_k<q\leq X_k+Y \\ f_1^2, f_2^2 | D(p,q)}}L(1,\chi_{d_1})L(1,\chi_{d_2})\log q \nonumber \\
+&O\Bigg(\sum_{k\in I}\sum_{\substack{f_1,f_2\leq 2\sqrt{X_k+Y}   }}\frac{1}{f_1f_2}
\sum_{\substack{X_k<q\leq X_k+Y \\ f_1^2, f_2^2 | D(p,q)}}L(1,\chi_{d_1})L(1,\chi_{d_2})\left(\frac{|D(p,X_k)|}{p^\frac{3}{2}}
+\frac{Y}{\sqrt{p}}\right)\Bigg).\label{hamicablesumtwo}
\end{align}
By the convexity bound $L(1,\chi_{d_i})\ll \log |d_i|\ll \log p$ for $i=1,2$ we have that the error term in $\eqref{hamicablesumtwo}$ is bounded by
\begin{equation}  \label{hamibound}
\frac{Y(\log p)^4}{p^\frac32}\sum_{k\in I}|D(p,X_k)|+Y(\log p)^4.
\end{equation}

Since $D(p,X_k)=0$ for $k$ on the end points of the interval $\displaystyle{\left[\frac{-2\sqrt{p}}{Y},
\frac{2\sqrt{p}}{Y}\right]\supseteq I}$, we have by the Euler-MacLaurin summation formula that
\begin{equation}
\sum_{k\in I} |D(p,X_k)|=\int_{\frac{-2\sqrt{p}}{Y}}^{\frac{2\sqrt{p}}{Y}}(4p-(tY)^2) dt 
+O\Bigg(\int_{\frac{-2\sqrt{p}}{Y}}^{\frac{2\sqrt{p}}{Y}} |t|Y^2 dt\Bigg)=\frac{32p^\frac{3}{2}}{3Y}+O\left(p\right) \label{sumofdpxk}.
\end{equation}
From $\eqref{hamibound}$ and $\eqref{sumofdpxk}$ we have that the error term in $\eqref{hamicablesumtwo}$ becomes $O(Y^2(\log p)^4)$. We then set $X:=X_k$ and define the inner sum in the main term of $\eqref{hamicablesumtwo}$ as
\begin{equation}\label{soneclasstime}
S_1:=\sum_{\substack{f_1,f_2\leq 2\sqrt{X+Y}  }}\frac{1}{f_1f_2}
\sum_{\substack{X<q\leq X+Y \\ f_1^2, f_2^2 | D(p,q)}}L(1,\chi_{d_1})L(1,\chi_{d_2})\log q.
\end{equation}

We have the following technical result for the sum in $\eqref{soneclasstime}$. We delay the proof until Section \ref{section: four}.
\begin{prop} \label{propdsseven}
Let $\epsilon,\gamma>0$. Suppose that $p^-<X<X+Y\leq p^+$ with $Y \gg \frac{\sqrt{p}}{(\log p)^\nu}$ for $\nu\geq 0$. Then we have that
\begin{multline*}
\sum_{\substack{f_1, f_2\leq 2\sqrt{X+Y}   }}
\frac{1}{f_1f_2}\sum_{\substack{X<q\leq X+Y \\ f_1^2, f_2^2 | D(p,q)}}L(1,\chi_{d_1})L(1,\chi_{d_2})\log q = C_2(p)Y+O\Bigg(\frac{Y}{(\log p)^\gamma}\\
+\sum_{f_1,f_2\leq (\log 4p)^{12+4\nu+4\gamma}}f_1f_2\sum_{n_1,n_2\leq p^{\epsilon}}E(X,Y;4[n_1f_2^2,n_2f_2^2])\Bigg),
\end{multline*}
where $C_2(p)$ is defined in $\eqref{actualkzerop}$.
\end{prop}

The result now follows from applying Proposition \ref{propdsseven} with $\nu = 5+\gamma$ and $\eqref{sumofdpxk}$ to the main term in  $\eqref{hamicablesumtwo}$ and  $\eqref{hamibound}$ and applying the fact that $|D(p,X_k)|\ll p$ to the error terms.
\end{proof}

We now use Lemma \ref{lemthreefour} and Theorem \ref{thrmdsthree} to prove the main result of this paper.
\begin{proof}(Proof of Theorem \ref{maintheorem})
We first apply the result from Theorem \ref{thrmdsthree} in Theorem \ref{aliquotaverage}, which gives
\begin{multline}
\frac{1}{|\C|}\sum_{E\in \C}\pi_{E,2}(X)=\frac{8}{3\pi^2}\sum_{p\leq X}\frac{C_2(p)}{\sqrt{p}(\log p)}\left(1+O(X^{-\epsilon})\right)\\
+O\left(\sum_{p\leq X}\frac{1}{p\log p}\sum_{\frac{-2\sqrt{p}}{Y}<k\leq \frac{2\sqrt{p}}{Y}}\sum_{f_1,f_2\leq (\log 4p)^{32+8\gamma}}f_1f_2\sum_{n_1,n_2\leq p^{\epsilon}}E(X_k,Y;4[n_1f_2^2,n_2f_2^2])\right).
\label{eqn: avgc2p}
\end{multline}
Now the remainder of the proof is reduced to showing that
\begin{align}
\sum_{p\leq X}\frac{1}{p\log p}\sum_{\frac{-2\sqrt{p}}{Y}<k\leq \frac{2\sqrt{p}}{Y}}\sum_{f_1,f_2\leq (\log 4p)^{32+8\gamma}}f_1f_2\sum_{n_1,n_2\leq p^{\epsilon}}E(X_k,Y;4[n_1f_2^2,n_2f_2^2])\ll \frac{\sqrt{X}}{(\log X)^{2+\epsilon}}.\label{errorofp}
\end{align}
We first bound the inner sum in $\eqref{errorofp}$, which gives
\begin{align*}
\sum_{f_1,f_2\leq (\log 4p)^{32+8\gamma}}f_1f_2\sum_{n_1,n_2\leq p^{\epsilon}}E(X_k,Y;4[n_1f_2^2,n_2f_2^2])\ll(\log 4p)^{64+16\gamma}\sum_{m\leq p^{3\epsilon}}E(X_k,Y;m).
\end{align*}
Next we set $$j:=\frac{(p+1)}{Y}+k\quad\text{ and }\quad f(p):=\frac{(\log p)^{63+16\gamma}}{p},$$ 
and by extending the sum over primes to a sum over all integers we have that 
the left hand side of $\eqref{errorofp}$ is bounded by
$$\sum_{p\leq X}f(p)\sum_{\frac{p^-}{Y}<j\leq \frac{p^+}{Y}}\sum_{m\leq p^{3\epsilon}}E(jY,Y;m)\leq\sum_{n\leq X}f(n)\sum_{\frac{n^-}{Y}<j\leq \frac{n^+}{Y}}\sum_{m\leq n^{3\epsilon}}E(jY,Y;m).$$
 We break up the sum into dyadic intervals which gives
\begin{align*}
&\sum_{\frac{X}{2}<n\leq X}f(n)\sum_{\frac{n^-}{Y}<j\leq \frac{n^+}{Y}}\sum_{m\leq n^{3\epsilon}}E(jY,Y;m)\ll f(X)\sum_{m\leq X^{3\epsilon}}\sum_{j\asymp \frac{X}{Y}}E(jY,Y;m) \sum_{\substack{\frac{X}{2}<n\leq X \\ (jY)^-<n\leq (jY)^+}}1\\
\ll&f(X)\sqrt{X}\sum_{m\leq X^{3\epsilon}}\sum_{j\asymp \frac{X}{Y}}E(jY,Y;m):=E'.
\end{align*}
Let $B>0$ and set $Y':=\frac{Y}{(\log X)^{B+4}}$. Then we have that 
\begin{align}
E'\ll &f(X)\sqrt{X}\sum_{m\leq X^{3\epsilon}}\sum_{j\asymp \frac{X}{Y}}\left(\frac{1}{Y'}\int_{jY}^{jY+Y'}E(u,Y;m)du+\frac{Y'}{m}\right)\nonumber\\
\ll&\frac{f(X)\sqrt{X}}{Y'}\sum_{m\leq X^{3\epsilon}}\int_{c_1X}^{c_2X}E(u,Y;m)du+\frac{f(X)X^{\frac{3}{2}}}{(\log X)^{B+3}}.\label{eprimerror}
\end{align}
The result follows by applying Lemma \ref{lemthreefour} to $\eqref{eprimerror}$ with $B=62+17\gamma$ and $A=126+32\gamma$.
\end{proof}

\section{Proof of Proposition \ref{propdsseven}}\label{section: four}
In this section we give the proof of Proposition \ref{propdsseven} stated in the previous section. Expanding upon the techniques of \cite[Theorem 7]{CDES:1} and \cite[Proposition 3.2]{JP:1}, we determine the coefficient of the main term for $S_1$ in $\eqref{soneclasstime}$ as a function of $p$.

\begin{proof} (Proof of Proposition \ref{propdsseven})
Recall from $\eqref{soneclasstime}$ that 
\begin{equation}
\label{sonenew}
S_1:=\sum_{\substack{f_1,f_2\leq 2\sqrt{X+Y}  }}\frac{1}{f_1f_2}
\sum_{\substack{X<q\leq X+Y \\ f_1^2, f_2^2 | D(p,q)}}L\left(1,\chi_{d_1}\right)
L\left(1,\chi_{d_2}\right)\log q.
\end{equation}

For the duration of this section let $z:=\log(4p)$ and let $\alpha$ be a parameter $\geq 10$. Now let $S_1'$ denote the double sum on the right hand side of $\eqref{sonenew}$ with  $\displaystyle{L\left(1,\chi_{d_i};z^{8\alpha^2}\right)}$
in place of $L\left(1,\chi_{d_i}\right)$ for $i=1,2$.  We estimate the error term $S_1-S_1'$ by applying Proposition \ref{lemtwothree} with $Q=4p$. We have that $0\leq -D(p,q)\leq 4p$ for $q\in (p^-,p^+)$.  Moreover, $\QQ\left(\sqrt{\frac{D(p,q)}{f_i^2}}\right)=\QQ(\sqrt{D(p,q)})$ and hence the conductor of $\chi_{d_i}$ is equal to ${\rm disc}(\QQ(\sqrt{D(p,q)}))$ for $i=1,2$. If ${\rm disc}(\QQ(\sqrt{D(p,q)})) \not \in {\mathcal{E}}_\alpha(4p)$ then by Proposition \ref{lemtwothree} and Mertens' theorem,
\begin{align*}
&L(1,\chi_{d_1})L(1,\chi_{d_2})-L(1,\chi_{d_1};z^{8\alpha^2})L(1,\chi_{d_2};z^{8\alpha^2})
\ll_\alpha \frac{(\log z)^2}{z^\alpha}.
\end{align*}
 If ${\rm disc}(\QQ(\sqrt{D(p,q)})) \in {\mathcal{E}}_\alpha(4p)$ then  for $i=1,2,$ we use the convexity bound 
$L\left(1,\chi_{d_i}\right)\ll \log p$. In this case we have that 
\begin{equation*}
L(1,\chi_{d_1})L(1,\chi_{d_2})-L(1,\chi_{d_1};z^{8\alpha^2})L(1,\chi_{d_2};z^{8\alpha^2}) 
\ll_\alpha z^2,
\end{equation*}
and thus
\begin{align*}
S_1-S_1' &\ll_\alpha \frac{(\log z)^2}{z^\alpha}\sum_{\substack{f_1, f_2\leq 2\sqrt{X+Y}  }}
\frac{1}{f_1f_2}\sum_{\substack{X<q\leq X+Y \\ f_1^2, f_2^2 | D(p,q) \\ {\rm disc}(\QQ(\sqrt{D(p,q)})) \not \in {\mathcal{E}}_\alpha(4p)}}\log q \\
&+z^2\sum_{\substack{f_1, f_2\leq 2\sqrt{X+Y}  }}
\frac{1}{f_1f_2}\sum_{\substack{X<q\leq X+Y \\ f_1^2, f_2^2 | D(p,q) \\ {\rm disc}(\QQ(\sqrt{D(p,q)})) \in {\mathcal{E}}_\alpha(4p)}}\log q.
\end{align*}
For $q \in (p^-,p^+)$ such that $\Delta:= {\rm disc}(\QQ(\sqrt{D(p,q)})) \in {\mathcal{E}}_\alpha(4p)$ we have that $D(p,q)=\Delta m^2$ for 
some $m \in \NN$. Equivalently $(p+1-q)^2-\Delta m^2=4p,$ where $\Delta\equiv D(p,q)\equiv 1 \imod{4}$. Let $n=p+1-q$, 
then for a fixed $\Delta \in {\mathcal{E}}_\alpha(4p)$ we need to determine the quantity
\begin{align*}
r(4p,2):=&\#\{(m,n)\in \ZZ^2: n^2-\Delta m^2=4p \}, \\
=&\#\left\{\frac{n+m\sqrt{\Delta}}{2}\in {\mathcal{O}}_K : N\left(\frac{n+m\sqrt{\Delta}}{2}\right)=p\right\},
\end{align*}
where $K=\QQ(\sqrt{\Delta})=\QQ(\sqrt{D(p,q)}),{\mathcal{O}}_K $ is its ring of integers, and $N(\cdot)$ is the norm of an element in $K$. Note that $$\#\{I\subseteq {\mathcal{O}}_K: N(I)=d\}=\sum_{k\mid d} \left(\frac{\Delta}{k}\right),$$

where $N(I)$ denotes the norm of an ideal $I\subseteq {\mathcal{O}}_K$. Thus, $$\frac{r(4p,2)}{6}\leq \#\{I\subseteq {\mathcal{O}}_K: N(I)=p\}=\sum_{k\mid p} \left(\frac{\Delta}{k}\right)\leq2$$
and hence $r(4p,2)\leq 12.$
Since there are at most 12 admissible pairs $(m,n)$ there are at most 12 admissible values of $q$ since $p$ is fixed. We have that $\alpha\geq 10$, and therefore from Proposition \ref{lemtwothree} we have that
$$\#\{p^-< q< p^+ : {\rm disc}(\QQ(\sqrt{D(p,q)})) \in {\mathcal{E}}_\alpha(4p)\} \leq 12\#{\mathcal{E}}_\alpha(4p)\ll p^\frac{1}{5}.$$ From the bound
\begin{equation} \label{theboundonfs}
 \sum_{\substack{f_1, f_2\leq 2\sqrt{X+Y}  }}
\frac{1}{f_1f_2}\sum_{\substack{X<q\leq X+Y \\ f_1^2, f_2^2 | D(p,q)}}\log q 
\ll Y(\log p)^3,
\end{equation}
we conclude that
\begin{equation}\label{eqn: soneminus}
S_1-S_1'\ll_\alpha \frac{(\log z)^2}{z^\alpha}Y(\log p)^3+z^2 p^\frac{1}{5} (\log p)^3
\ll_{\alpha} Y\left(\frac{(\log p)^3(\log \log 4p)^2}{(\log 4p)^{\alpha}}\right). 
\end{equation}
Let $u\geq 1$ be a parameter to be determined later and for convenience let $y:=z^{8\alpha^2}$. We have that $$L(1,\chi;y) = \sum_{P^+(n)\leq y} \frac{\chi(n)}{n} = \sum_{\substack{P^+(n)\leq y \\ n\leq y^u}} 
\frac{\chi(n)}{n} + O\Bigg(\sum_{\substack{P^+(n)\leq y \\ n> y^u}} \frac{1}{n} \Bigg).$$
Note that for the error term above we have that
\begin{align*}
\sum_{\substack{P^+(n)\leq y \\ n> y^u}} \frac{1}{n} \leq &\frac{1}{e^u} \sum_{\substack{P^+(n)\leq y \\ n> y^u}} \frac{1}{n^{1-\frac{1}{\log y}}}\leq\frac{1}{e^u}\prod_{p\leq y}\left(1-\frac{1}{p^{1-
\frac{1}{\log y}}}\right)^{-1}=\frac{1}{e^u}\prod_{p\leq y}\left(1-\frac{1}{p}+O\left(\frac{\log p}{p\log y}\right)\right)^{-1}\\
\ll&\frac{1}{e^u}\prod_{p\leq y}\left(1-\frac{1}{p}\right)^{-1} \prod_{p\leq y}\left(1+O\left(\frac{ \log p}{(p-1)\log y}\right)\right)^{-1}\ll \frac{\log y}{e^u},
\end{align*}
since the sum $\sum_{p\leq y} \frac{ \log p}{(p-1)\log y}$ converges. Then
\begin{align}\label{lprodnew}
 L(1,\chi_{d_1};y)L(1,\chi_{d_2};y)=&\sum_{\substack{P^+(n_1),P^+(n_2)\leq y \\ n_1,n_2 \leq y^{u}}}
\frac{\chi_{d_1}(n_1)\chi_{d_2}(n_2)}{n_1n_2}
 +O\Bigg(\frac{\log y}{e^u}\sum_{\substack{P^+(n)\leq y \\ n\leq y^{u}}}\frac{1}{n}+ \frac{(\log y)^2}{e^{2u}}\Bigg)
\end{align}
and by $\eqref{theboundonfs}$, $\eqref{eqn: soneminus}$ and $\eqref{lprodnew}$, we have that $\eqref{sonenew}$  becomes 
\begin{align}
S_1&= \sum_{\substack{f_1,f_2\leq 2\sqrt{X+Y} }}\frac{1}{f_1f_2}
\sum_{\substack{X<q\leq X+Y \\ f_1^2, f_2^2 | D(p,q)}}\log q \sum_{\substack{P^+(n_1),P^+(n_2)\leq y \\ n_1,n_2 \leq y^u}}
\frac{\chi_{d_1}(n_1)\chi_{d_2}(n_2)}{n_1n_2}\nonumber \\
&+O_\alpha\left(\frac{Y(\log \log 4p)^2}{(\log p)^{\alpha-3}}+ \frac{Yu(\log p)^3(\log y)^2}{e^u}\right).\label{letsfinishtwo}
\end{align}
Let $V$ be a parameter to be determined later such that $1\leq V\leq 2\sqrt{X+Y}$. We write the main term in $\eqref{letsfinishtwo}$ as
\begin{align}
&\sum_{\substack{f_1,f_2\leq V  }}\frac{1}{f_1f_2}\sum_{\substack{P^+(n_1),P^+(n_2)\leq y \\ n_1,n_2 \leq y^u}}\frac{1}{n_1n_2}
\sum_{\substack{X<q\leq X+Y \\ f_1^2, f_2^2 | D(p,q)}}\chi_{d_1}(n_1)\chi_{d_2}(n_2)\log q + \Bigg(\sum_{\substack{f_1\leq V \\ V\leq f_2\leq 2\sqrt{X+Y}  }}\nonumber\\
+&\sum_{\substack{V<f_1\leq 2\sqrt{X+Y} \\ f_2\leq V  }}+\sum_{\substack{V<f_1,f_2\leq 2\sqrt{X+Y}  }}\Bigg)\frac{1}{f_1f_2}
\sum_{\substack{P^+(n_1),P^+(n_2)\leq y \\ n_1,n_2 \leq y^u}}\frac{1}{n_1n_2}\sum_{\substack{X<q\leq X+Y \\ f_1^2, f_2^2 | D(p,q)}}
\chi_{d_1}(n_1)\chi_{d_2}(n_2)\log q\nonumber\\
:=&M_1 +E_1.\label{newsumone}
\end{align}
From Lemma \ref{dslemtwelve} and the definition of $Y$ we have that the inner sum in $E_1$ in $\eqref{newsumone}$ is
\begin{align}
&\sum_{\substack{P^+(n_1),P^+(n_2)\leq y \\ n_1,n_2 \leq y^u}}
\frac{1}{n_1n_2}\sum_{\substack{X<q\leq X+Y \\ f_1^2, f_2^2 | D(p,q)}}\chi_{d_1}(n_1)\chi_{d_2}(n_2)\log q \nonumber\\
\ll&\log p\sum_{\substack{P^+(n_1),P^+(n_2)\leq y \\ 
n_1,n_2 \leq y^u}}\frac{1}{n_1n_2}\sum_{\substack{X<k\leq X+4\sqrt{p} \\ f_1, f_2 | D(p,k)}} 1\nonumber\\
\ll&(\log p)\sqrt{p}\sum_{\substack{P^+(n_1),P^+(n_2)\leq y \\ n_1,n_2 \leq y^u}}
\frac{\:\#\{m\in \ZZ/[f_1,f_2]\ZZ : D(p,m)\equiv 0 \imod{[f_1,f_2]}\}}{n_1n_2[f_1,f_2]}\nonumber\\
\ll&\frac{\sqrt{p}(\log p)u^2 (\log y)^2}{\sqrt{[f_1,f_2]}}. \label{boundforvs}
\end{align}

Hence, from $\eqref{boundforvs}$ we have that 
\begin{align}
E_1\ll&\sqrt{p}(\log p)u^2 (\log y)^2\Bigg(\sum_{\substack{f_1\leq V \\ V\leq f_2\leq 2\sqrt{X+Y}  }}+\sum_{\substack{V<f_1\leq 2\sqrt{X+Y} \\ f_2\leq V  }}+\sum_{\substack{V<f_1,f_2\leq 2\sqrt{X+Y}  }}\Bigg)\frac{\sqrt{(f_1,f_2)}}{(f_1f_2)^\frac{3}{2}}\nonumber \\
\ll&\frac{\sqrt{p}(\log p)u^2 (\log y)^2}{V^{\frac{1}{4}}},\label{perrorone}
\end{align}
by using the bound $(f_1,f_2)\ll \sqrt{f_1f_2}$. 

Combining the bounds from $\eqref{perrorone}$ and $\eqref{newsumone}$ with $\eqref{letsfinishtwo}$ gives
\begin{align}
S_1=&\sum_{\substack{f_1,f_2\leq V  }}\frac{1}{f_1f_2}
\sum_{\substack{P^+(n_1),P^+(n_2)\leq y \\ n_1,n_2 \leq y^u}}\frac{1}{n_1n_2}
\sum_{\substack{X<q\leq X+Y \\ f_1^2, f_2^2 | D(p,q)}}\chi_{d_1}(n_1)\chi_{d_2}(n_2)\log q  \nonumber\\
&+O_\alpha\left(\frac{Y(\log \log 4p)^2}{(\log p)^{\alpha-3}}+ \frac{Yu(\log p)^3(\log y)^2}{e^u} + \frac{\sqrt{p}(\log p)u^2(\log y)^2}{V^{\frac{1}{4}}}\right).\label{psumfour}
\end{align}
Then by quadratic reciprocity we have that the main term in $\eqref{psumfour}$ becomes
\begin{align}
&\sum_{\substack{f_1,f_2\leq V  }}\frac{1}{f_1f_2}
\sum_{\substack{P^+(n_1),P^+(n_2)\leq y \\ n_1,n_2 \leq y^u}}\frac{1}{n_1n_2}\sum_{\substack{a_1\in\ZZ/4n_1\ZZ \\ a_2\in\ZZ/4n_2\ZZ \\ 
a_1\equiv a_2\equiv 1\imod{4}}}\left(\frac{a_1}{n_1}\right)\left(\frac{a_2}{n_2}\right)
\sum_{\substack{X<q\leq X+Y \\ D(p,q)\equiv a_1f_1^2\imod{4n_1f_1^2} \\ D(p,q) \equiv a_2f_2^2\imod{4n_2f_2^2}}}\log q\nonumber\\
=&\sum_{\substack{f_1,f_2\leq V  }}\frac{1}{f_1f_2}
\sum_{\substack{P^+(n_1),P^+(n_2)\leq y \\ n_1,n_2 \leq y^u}}\frac{1}{n_1n_2}\sum_{\substack{a_1\in\ZZ/4n_1\ZZ \\ a_2\in\ZZ/4n_2\ZZ \\ 
a_1\equiv a_2\equiv 1\imod{4}}}\left(\frac{a_1}{n_1}\right)\left(\frac{a_2}{n_2}\right)\nonumber \\
\times&\Bigg(\sum_{\substack{b\in(\ZZ/4[n_1f_1^2,n_2f_2^2]\ZZ)^* \\ D(p,b)\equiv a_1f_1^2\imod{4n_1f_1^2}  \\D(p,b)\equiv a_2f_2^2\imod{4n_2f_2^2}}}\sum_{\substack{X<q\leq X+Y \\ q\equiv b \imod{4[n_1f_1^2,n_2f_2^2]}}}\log q+\sum_{\substack{X<q\leq X+Y \\  D(p,q)\equiv a_1f_1^2\imod{4n_1f_1^2} \\ D(p,q) \equiv a_2f_2^2\imod{4n_2f_2^2} \\ q\mid 4[n_1f_1^2,n_2f_2^2]}}\log q \Bigg).\label{qdividesl}
\end{align}

Since $q\neq 2$, if $q\mid 4[n_1f_1^2,n_2f_2^2]$ then either $q\mid n_1f_1^2 $ or $q\mid n_2f_2^2$. If $q\mid n_1f_1^2 $ then since $D(p,q)=q^2-2(p+1)q+(p-1)^2 \equiv a_1f_1^2 \imod{ 4n_1f_1^2} $ then $q\mid (4n_1f_1^2,(p-1)^2-a_1f_1^2)$. This implies that $q\mid n_1(p-1)$. Similarly, if $q\mid n_2f_2^2 $ then $q\mid n_2(p-1)$. Thus, the second sum in $\eqref{qdividesl}$ is bounded by
\begin{align}
&\sum_{f_1,f_2\leq V}\frac{1}{f_1f_2}\sum_{n_1,n_2 \leq y^u}
\left(\sum_{q\mid n_1} \log q + \sum_{q\mid n_2} \log q+\sum_{q\mid p-1} \log q\right) \nonumber \\
\ll&\sum_{f_1,f_2\leq V }\frac{1}{f_1f_2}\sum_{n_1,n_2 \leq y^u}\log(n_1n_2(p-1))\ll 
y^{2u}(\log V)^2(u\log y+\log p)\label{perrorthree}.
\end{align}

Let $L:=4[n_1f_1^2,n_2f_2^2]$ then we replace the first inner sum in $\eqref{qdividesl}$ by 
$$\sum_{\substack{X<q\leq X+Y \\ q\equiv b \imod{L}}}\log q:=\frac{Y}{\varphi(L)}+\left(\sum_{\substack{X<q\;\leq X+Y \\ q\equiv b \imod{L}}} \log q-\frac{Y}{\varphi(L)}\right).$$ If we fix $b\in(\ZZ/L\ZZ)^*$ there is at most one pair $(a_1,a_2) \in \ZZ/4n_1\ZZ\times \ZZ/4n_2\ZZ$ such that
$D(p,b)\equiv a_if_i^2\imod{4n_if_i^2}$ for $i=1,2$. Hence, we have that
\begin{align}
 &\sum_{\substack{a_1\in\ZZ/4n_1\ZZ \\ a_2\in\ZZ/4n_2\ZZ \\ a_1\equiv a_2\equiv 1\imod{4}}}\left(\frac{a_1}{n_1}\right)
 \left(\frac{a_2}{n_2}\right)\sum_{\substack{b\in(\ZZ/L\ZZ)^* \\ D(p,b)\equiv a_1f_1^2\imod{4n_1f_1^2}  \\D(p,b)\equiv a_2f_2^2\imod{4n_2f_2^2}}}\Bigg(\sum_{\substack{X<q\;\leq X+Y \\ q\equiv b \imod{L}}} \log q-\frac{Y}{\varphi(L)}\Bigg)\nonumber\\
 \ll&\sum_{b\in (\ZZ/L\ZZ)^*}\Bigg|\sum_{\substack{X<q\;\leq X+Y \\ q\equiv b \imod{L}}} \log q-\frac{Y}{\varphi(L)}\Bigg|\leq \varphi(L)E(X,Y;L), \label{errorboundxy}
\end{align}
where the definition of $E(X,Y;q)$ is given in $\eqref{eerror}$. From $\eqref{qdividesl}$, $\eqref{perrorthree}$ and $\eqref{errorboundxy}$ we have that $\eqref{psumfour}$ becomes
\begin{align}
S_1=&Y\sum_{\substack{f_1,f_2\leq V  }}\frac{1}{f_1f_2}
\sum_{\substack{P^+(n_1)\leq y \\ P^+(n_2)\leq y \\ n_1,n_2 \leq y^u}}\frac{1}{\varphi(L)n_1n_2}\sum_{\substack{a_1\in\ZZ/4n_1\ZZ \\ a_2\in\ZZ/4n_2\ZZ \\ a_1\equiv a_2\equiv 1\imod{4}}}\left(\frac{a_1}{n_1}\right)\left(\frac{a_2}{n_2}\right)\sum_{\substack{b\in(\ZZ/L\ZZ)^* \\ D(p,b)\equiv a_1f_1^2\imod{4n_1f_1^2}  \\D(p,b)\equiv a_2f_2^2\imod{4n_2f_2^2}}}1  \nonumber\\
&+O_\alpha\Bigg(\frac{Y(\log \log 4p)^2}{(\log p)^{\alpha-3}}+ \frac{Yu(\log p)^3(\log y)^2}{e^u} + \frac{\sqrt{p}(\log p)u^2(\log y)^2}{V^{\frac{1}{4}}}\nonumber\\
&+y^{2u}(\log V)^2(u\log y+\log p)+\sum_{\substack{f_1,f_2\leq V}}f_1f_2\sum_{\substack{n_1,n_2 \leq y^u}}E(X,Y;L)\Bigg).\label{psumfive}
\end{align}

Let $\ell$ be a prime then by the Chinese remainder theorem we have that the inner sum in the main term of $\eqref{psumfive}$ becomes
\begin{multline}
\sum_{\substack{a_1\in\ZZ/4n_1\ZZ \\ a_2\in\ZZ/4n_2\ZZ \\ a_1\equiv a_2\equiv 1\imod{4}}}\left(\frac{a_1}{n_1}\right)\left(\frac{a_2}{n_2}\right)\prod_{\ell \mid L}\#\Big\{m\in (\ZZ/\ell^{\nu_\ell(L)}\ZZ)^*:D(p,m)\equiv a_1f_1^2\imod{\ell^{\nu_\ell(4n_1f_1^2)}}\\
\text{ and }D(p,m) \equiv a_2f_2^2\imod{\ell^{\nu_\ell(4n_2f_2^2)}}\Big\}\label{sumdpm}.
\end{multline}
Note that the set in the product of $\eqref{sumdpm}$ is empty unless $$a_1f_1^2\equiv a_2f_2^2\imod{(4n_1f_1^2,4n_2f_2^2)}.$$ 

For convenience, we give the notation,
\begin{align}
(a,f,n) :=&\begin{cases}(a_1,f_1,n_1) & {\rm if}\:\max\{\nu_\ell(4n_1f_1^2),\nu_\ell(4n_2f_2^2)\}=\nu_\ell(4n_1f_1^2), \\ (a_2,f_2,n_2) & {\rm if}\:\max\{\nu_\ell(4n_1f_1^2),\nu_\ell(4n_2f_2^2)\}=\nu_\ell(4n_2f_2^2).  \end{cases}
\label{afn}
\end{align}
Then we write $\eqref{sumdpm}$ as
\begin{align}
\sum_{\substack{a_1\in\ZZ/4n_1\ZZ, a_2\in\ZZ/4n_2\ZZ \\ a_1\equiv a_2\equiv 1\imod{4} \\ a_1f_1^2\equiv a_2f_2^2\imod{(4n_1f_1^2,4n_2f_2^2)}}}\left(\frac{a_1}{n_1}\right)\left(\frac{a_2}{n_2}\right)\prod_{\ell \mid L}\cp,
\label{cplsum}
\end{align}
where 
$$\cp:=\#\{m\in \left(\ZZ/\ell^{\nu_\ell(4nf^2)}\ZZ\right)^*: D(p,m)\equiv af^2\imod{\ell^{\nu_\ell(4nf^2)}}\}.$$ 
A more general version of the sum $\cp$ where $p$ is any odd integer was first considered in \cite{CDES:1}. Since we only consider when $p$ is prime, we have the following special case of \cite[Lemma 10]{CDES:1}. 

\begin{lem} \label{lemdsten}
Let $p$ be an odd prime, let $f$ be odd and let $a\equiv 1 \imod{4}$. Let $\ell$ be any odd prime dividing $nf$, and let $e=\nu_\ell(4nf^2)=\nu_\ell(nf^2).$ If $\ell \nmid 4p+af^2,$ then
\begin{equation*}
\cp=\begin{cases} 1+\left(\frac{4p+af^2}{\ell}\right) & \text{ if } \ell \nmid (p-1)^2-af^2, \\ 1 & \text{ if }  \ell \mid (p-1)^2-af^2.\end{cases}
\end{equation*}
If $\ell \mid 4p+af^2$, then, with $s=\nu_\ell(4p+af^2)$, we have
\begin{equation*}
\cp=\begin{cases} 2\left(\frac{p+1}{\ell}\right)^2\ell^\frac{s}{2} & \text{ if }  1\leq s < e, 2\mid s, \text{ and }\left(\frac{(4p+af^2)/\ell^s}{\ell}\right)=1,\\ \left(\frac{p+1}{\ell}\right)^2\ell^{\lfloor e/2 \rfloor} & \text{ if }  s\geq e, \\ 0 &\text{  otherwise}.\end{cases}
\end{equation*}
In particular, if $\ell \mid f$, then
\begin{equation*}
\cp:=\cpf=\begin{cases} 1 +\left(\frac{p(p-1)^2}{\ell}\right) & \text{ if }  \ell\neq p, \\  0 &\text{ otherwise}.\end{cases}
\end{equation*}
Furthermore,
\begin{equation*}
C_p^{(2)}(a,f,n)=\begin{cases} 2 & \text{ if } \nu_2(4nf^2)=2+\nu_2(n)=2,  \\  4 & \text{ if }  \nu_2(4nf^2)=2+\nu_2(n)\geq 3\text{ and } a\equiv 5 \imod{8},  \\0 &\text{ otherwise}.\end{cases}
\end{equation*}
\end{lem}

Since $C_p^{(2)}(a,f,n)$ does not depend on $f$, for convenience we define
\begin{equation*}
S^{(2)}(a,n):=\frac{C_p^{(2)}(a,f,n)}{2}=\begin{cases} 1 & \text{ if } 2\nmid n,  \\  2 & \text{ if }  2\mid n\text{ and } a\equiv 5 \imod{8},  \\0 &\text{ otherwise}.\end{cases}
\end{equation*}
We now break up the product in $\eqref{cplsum}$ as 
\begin{equation}\label{deltacp}
\prod_{\ell \mid L}\cp=S^{(2)}(a,n)\prod_{\substack{\ell\mid n_1n_2 \\ \ell\neq 2}}\cp\prod_{\substack{\ell \mid f_1f_2 \\ \ell\nmid n_1n_2}}\cp.
\end{equation}
Note that if $\ell \mid f_1f_2$ and $\ell \nmid 2n_1n_2$ then $\max\{\nu_\ell(4n_1f_1^2),\nu_\ell(4n_2f_2^2)\}=\max\{\nu_\ell(f_1^2),\nu_\ell(f_2^2)\}>0$ and hence
\begin{align*}
(a,f,n) =&\begin{cases}(a_1,f_1,n_1) & {\rm if}\:\max\{\nu_\ell(f_1^2),\nu_\ell(f_2^2)\}=\nu_\ell(f_1^2), \\ (a_2,f_2,n_2) & {\rm if}\:\max\{\nu_\ell(f_1^2),\nu_\ell(f_2^2)\}=\nu_\ell(f_2^2).  \end{cases}
\end{align*}
From Lemma \ref{lemdsten} and $\eqref{deltacp}$ we have that
\begin{equation}
\sum_{\substack{a_1\in(\ZZ/4n_1\ZZ)^* \\a_2\in(\ZZ/4n_2\ZZ)^* \\ a_1\equiv a_2\equiv 1\imod{4}}}\left(\frac{a_1}{n_1}\right)\left(\frac{a_2}{n_2}\right)\sum_{\substack{b\in(\ZZ/L\ZZ)^* \\ D(p,b)\equiv a_1f_1^2\imod{4n_1f_1^2}  \\D(p,b)\equiv a_2f_2^2\imod{4n_2f_2^2}}}1=2C_{p,f}(n_1,n_2)\prod_{\substack{\ell \mid f_1f_2 \\ \ell \nmid n_1n_2}}\cpf,\label{newsumdpm}
\end{equation}
where we define
\begin{align}
C_{p,f}(n_1,n_2):=&\sum_{\substack{a_1\in(\ZZ/4n_1\ZZ)^*, a_2\in(\ZZ/4n_2\ZZ)^* \\ a_1\equiv a_2\equiv 1\imod{4}\\ a_1f_1^2\equiv a_2f_2^2\imod{(4n_1f_1^2,4n_2f_2^2)}}}\left(\frac{a_1}{n_1}\right)\left(\frac{a_2}{n_2}\right)S^{(2)}(a,n)\prod_{\substack{\ell\mid n_1n_2 \\ \ell \neq 2}}\cp.
\end{align}
Then from $\eqref{newsumdpm}$ we can express the main term of $\eqref{psumfive}$ as 

\begin{align}
&Y\sum_{\substack{f_1,f_2\leq V  }}\frac{1}{f_1f_2}\sum_{\substack{P^+(n_1),P^+(n_2)\leq y \\ n_1,n_2 \leq y^u}}\frac{2C_{p,f}(n_1,n_2)}{\varphi(L)n_1n_2}
\prod_{\substack{\ell \mid f_1f_2 \\ \ell \nmid n_1n_2}}\cpf  \nonumber\\
=&Y\sumv\frac{\prod_{\ell \mid f_1f_2}\cpf}{f_1f_2}
\sum_{\substack{P^+(n_1),P^+(n_2)\leq y \\ n_1,n_2 \leq y^u}}\frac{2C_{p,f}(n_1,n_2)}{\varphi(L)n_1n_2}\prod_{\ell \mid (f_1f_2, n_1n_2)}\cpf^{-1}\label{mainsone},
\end{align}
where the prime on the sum over $f_1$ and $f_2$ indicates that the sums are to be restricted to those $f_1,f_2$ such that $\cpf\neq 0$. Now, we consider the size of $C_{p,f}(n_1,n_2)$ in the following lemma. We delay the proof until Section 5.

\begin{lem} \label{lemdseleven}
Let $p,f_1,f_2$, and $f$ be odd integers. Then the function $C_{p,f}(n_1,n_2)$ is multiplicative in $n_1,n_2$. Let $\alpha_1,\alpha_2$ be non-negative integers and let $\ell$ be an odd prime. Then $$C_{p,f}(2^{\alpha_1},2^{\alpha_2})=2^{\max\{\alpha_1,\alpha_2\}}(-1)^{\alpha_1+\alpha_2}.$$
If $\ell=p$ and  $\ell\nmid f_1f_2$, then
\begin{equation*}
C_{p,f}(p^{\alpha_1},p^{\alpha_2})=p^{\max\{\alpha_1,\alpha_2\}-1}(p-2).
\end{equation*}
If $\ell \nmid pf_1f_2$, then
\begin{equation*}
C_{p,f}(\ell^{\alpha_1},\ell^{\alpha_2})=\ell^{\max\{\alpha_1,\alpha_2\}-1}\begin{cases} \ell-1-\left(\frac{p}{\ell}\right)-\left(\frac{p-1}{\ell}\right)^2 & \text{ if }  2\mid \alpha_1+\alpha_2, \\
-1 -\left(\frac{p-1}{\ell}\right)^2 & \text{ if } 2\nmid\alpha_1+\alpha_2.\end{cases}
\end{equation*}
If $\ell \mid f_1f_2$, then
\begin{equation*}
\frac{C_{p,f}(\ell^{\alpha_1},\ell^{\alpha_2})}{\ell^{\max\{\alpha_1,\alpha_2\}-1}}=\cpf
\begin{cases} 
\ell-1 &\text{if }\alpha_2=0, 2\mid \alpha_1  \text{ and }\nu_\ell(f_1^2) \geq \nu_\ell(f_2^2), \\
\ell-1 &\text{if }\alpha_1=0, 2\mid \alpha_2 \text{ and }\nu_\ell(f_2^2) \geq \nu_\ell(f_1^2), \\
\ell-1 & \text{if }\alpha_1+\alpha_2>0, 2\mid \alpha_1+\alpha_2\text{ and }\nu_\ell(f_1^2)=\nu_\ell(f_2^2), \\
 0 & \text{otherwise}.\end{cases}
\end{equation*}
Furthermore, for any $n_1,n_2$ we have the bound $$C_{p,f}(n_1,n_2) \ll \frac{n_1n_2}{(n_1,n_2)\kappa_{2p}(n_1n_2)} \prod_{\ell \mid (f_1f_2, n_1n_2)}\cpf,$$
where $\kappa_{2p}(m)$ is defined in $\eqref{kappadef}$.
\end{lem}
The next step is to extend the sums in $\eqref{mainsone}$ to sums over all integers. We first focus on bounding the inner sum. We note that if $((n_1,n_2),(f_1^2,f_2^2))=1$ then $(n_1f_1^2,n_2f_2^2)=(n_1,n_2)(f_1^2,f_2^2)$. Otherwise, there exists a prime $\ell$ such that $\ell\mid ((n_1,n_2),(f_1^2,f_2^2))$. If $C_{p,f}(\ell^{\nu_\ell(n_1)},\ell^{\nu_\ell(n_2)})\neq 0$ then from Lemma \ref{lemdseleven} either $\nu_\ell(f_1^2)=\nu_\ell(f_2^2)$ or $\ell \nmid (n_1,n_2)$ and thus $(n_1f_1^2,n_2f_2^2)=(n_1,n_2)(f_1^2,f_2^2)$. This gives the bound
$$\frac{1}{\varphi(L)}=\frac{(4n_1f_1^2,4n_2f_2^2)}{\varphi(16n_1n_2f_1^2f_2^2)}=\frac{(4n_1,4n_2)(f_1^2,f_2^2)}{\varphi(16n_1n_2f_1^2f_2^2)}\ll
\frac{(n_1,n_2)(f_1^2,f_2^2)}{\varphi(n_1n_2)\varphi(f_1^2)\varphi(f_2^2)}.$$ Hence, from Lemma \ref{lemdseleven} we have that
\begin{equation}
\frac{C_{p,f}(n_1,n_2)}{\varphi(L)}\ll\frac{n_1n_2(f_1^2,f_2^2)}{\kappa_{2p}(n_1n_2)
\varphi(n_1n_2)\varphi(f_1^2)\varphi(f_2^2)}
\prod_{\ell \mid (f_1f_2, n_1n_2)}\cpf. \label{badphibound}
\end{equation}
From $\eqref{badphibound}$ we have that $\eqref{mainsone}$ becomes
\begin{align}
&Y\sumv\frac{\prod_{\ell \mid f_1f_2}\cpf}{f_1f_2}
\sum_{P^+(n_1),P^+(n_2)\leq y }\frac{2C_{p,f}(n_1,n_2)}{\varphi(L)n_1n_2}\prod_{\ell \mid (f_1f_2, n_1n_2)}\cpf^{-1}\nonumber\\
+&O\Bigg(Y\sumv\frac{(f_1^2,f_2^2)\prod_{\ell \mid f_1f_2}\cpf}{f_1f_2\varphi(f_1^2)\varphi(f_2^2)}\Bigg(\sum_{\substack{P^+(n_1),P^+(n_2)\leq y \\ n_1\leq y^u,n_2 > y^u}}+\sum_{\substack{P^+(n_1),P^+(n_2)\leq y \\ n_1> y^u,n_2 \leq y^u}}\nonumber\\
+&\sum_{\substack{P^+(n_1),P^+(n_2)\leq y \\ n_1,n_2 > y^u}}\Bigg)\frac{1}{\kappa_{2p}(n_1n_2)\varphi(n_1n_2)}\Bigg).\label{perrorsix}
\end{align}

We now consider the three inner sums in the error term of $\eqref{perrorsix}$. Let $d:=(n_1,n_2)$ and write $n_1=dn_1',n_2=dn_2'$. Since  $\nu_\ell(d^2n_1'n_2')\equiv \nu_\ell(n_1'n_2')\imod{2}$ we have that $\kappa_{2p}(d^2n_1'n_2')=\kappa_{2p}(n_1'n_2')= \kappa_{2p}(n_1')\kappa_{2p}(n_2')$. Thus, breaking up the three inner sums in the error term in $\eqref{perrorsix}$ into sums over $d$ gives
\begin{align}
&\Bigg(\sum_{\substack{P^+(n_1),P^+(n_2)\leq y \\ n_1\leq y^u,n_2 > y^u}}+\sum_{\substack{P^+(n_1),P^+(n_2)\leq y \\ n_1> y^u,n_2 \leq y^u}}+\sum_{\substack{P^+(n_1),P^+(n_2)\leq y \\ n_1,n_2 > y^u}}\Bigg) \frac{1}{\kappa_{2p}(n_1n_2)\varphi(n_1n_2)}\nonumber\\
\leq&\Bigg(\sum_{d\leq y^u}\frac{1}{\varphi(d)^2}\sum_{\substack{n_1'\leq \frac{y^u}{d} \\ n_2' > \frac{y^u}{d}}}+\sum_{d\leq y^u}\frac{1}{\varphi(d)^2}\sum_{\substack{n_1'> \frac{y^u}{d} \\ n_2' \leq \frac{y^u}{d}}}+\sum_{d\geq 1}\frac{1}{\varphi(d)^2}\sum_{n_1',n_2'> \frac{y^u}{d}}\Bigg) \frac{1}{\kappa_{2p}(n_1')\kappa_{2p}(n_2')\varphi(n_1')\varphi(n_2')}\nonumber\\
\leq&\Bigg(\sum_{d\leq y^u}\frac{1}{\varphi(d)^2}\sum_{\substack{n_1'\leq \frac{y^u}{d} \\ n_2' > \frac{y^u}{d}}}+\sum_{d\leq y^u}\frac{1}{\varphi(d)^2}\sum_{\substack{n_1'> \frac{y^u}{d} \\ n_2' \leq \frac{y^u}{d}}}\Bigg)\frac{1}{\kappa_{2p}(n_1')\kappa_{2p}(n_2')\varphi(n_1')\varphi(n_2')}\nonumber\\
+&\sum_{d\leq y^u}\frac{1}{\varphi(d)^2}\left(\sum_{n'> \frac{y^u}{d}} \frac{1}{\kappa_{2p}(n')\varphi(n')}\right)^2+\sum_{d> y^u }\frac{1}{\varphi(d)^2}\left(\sum_{n'\geq 1} \frac{1}{\kappa_{2p}(n')\varphi(n')}\right)^2 .
\label{kappaboundone}
\end{align}

From Lemma \ref{lemdseight} and partial summation we have that $\eqref{kappaboundone}$ is bounded by
\begin{equation*}
\sum_{d\leq y^u}\left(\frac{1}{\varphi(d)^2}\frac{\sqrt{d}}{y^{\frac{u}{2}}}+\frac{1}{\varphi(d)^2}\frac{d}{y^{u}}\right)+\sum_{d> y^u }\frac{1}{\varphi(d)^2} \ll \frac{u(\log y)(\log \log y^u)^2}{y^u}.
\end{equation*}
By Lemma \ref{lemdsten} there exists an $\epsilon>0$ such that
$$\prod_{\ell \mid f_1f_2 }\cpf \ll \prod_{\ell \mid f_1f_2}2 = 2^{\omega(f_1f_2)} \ll 2^\frac{\log f_1f_2}{\log \log f_1f_2}\ll f_1^\epsilon f_2^\epsilon,$$
where $\omega(n)$ denotes the number of prime divisors of $n$ and thus from $\eqref{perrorsix}$ we have that $\eqref{mainsone}$ becomes
\begin{multline}
Y\sumv\frac{\prod_{\ell \mid f_1f_2}\cpf}{f_1f_2}
\sum_{\substack{P^+(n_1)\leq y\\ P^+(n_2)\leq y}}\frac{2C_{p,f}(n_1,n_2)}{\varphi(L)n_1n_2}\prod_{\ell \mid (f_1f_2, n_1n_2)}\cpf^{-1}\\
+O\left(Y\frac{u(\log y)(\log \log y^u)^2}{y^u}\right) \label{newestsumone}.
\end{multline}
Now, as in $\eqref{perrorsix}$ we write $\eqref{newestsumone}$ as 
\begin{align}
&\sumv\frac{\prod_{\ell \mid f_1f_2}\cpf}{f_1f_2}
\sum_{n_1,n_2\geq 1}\frac{2C_{p,f}(n_1,n_2)}{\varphi(L)n_1n_2}\prod_{\ell \mid (f_1f_2, n_1n_2)}\cpf^{-1}\nonumber\\
+&O\Bigg(Y\sumv\frac{(f_1^2,f_2^2)\prod_{\ell \mid f_1f_2}\cpf}{f_1f_2\varphi(f_1^2)\varphi(f_2^2)}\Bigg(
\sum_{\substack{P^+(n_1)>y \\ P^+(n_2)\leq y}}+\sum_{\substack{P^+(n_1)\leq y \\ P^+(n_2)> y}}\nonumber\\
+&\sum_{\substack{P^+(n_1),P^+(n_2)> y}}\Bigg)\frac{1}{\kappa_{2p}(n_1n_2)\varphi(n_1n_2)}+Y\frac{u(\log y)(\log \log y^u)^2}{y^u}\Bigg)\label{newestsumtwo}.
\end{align}
Following analogously to $\eqref{kappaboundone}$ we have that the three inner sums in the error term in $\eqref{newestsumtwo}$ are bounded by 
\begin{equation}
\Bigg(\sum_{\substack{P^+(n_1)>y \\ P^+(n_2)\leq y}}+\sum_{\substack{P^+(n_1)\leq y \\ P^+(n_2)> y}}+\sum_{\substack{P^+(n_1),P^+(n_2)> y}}\Bigg)\frac{1}{\kappa_{2p}(n_1n_2)\varphi(n_1n_2)} \ll \frac{1}{\sqrt{y}}.\label{pplusboundtwo}
\end{equation}
Thus, from $\eqref{pplusboundtwo}$ we can write $\eqref{newestsumtwo}$ as
 \begin{align}
&Y\sumfs\frac{\prod_{\ell \mid f_1f_2}\cpf}{f_1f_2}
\sum_{n_1,n_2\geq 1}\frac{2C_{p,f}(n_1,n_2)}{\varphi(L)n_1n_2}\prod_{\ell \mid (f_1f_2, n_1n_2)}\cpf^{-1}\nonumber\\
+&O\Bigg(Y\Bigg(\sum_{\substack{f_1\leq V\\f_2> V}}+\sum_{\substack{f_1> V\\f_2\leq V}}+\sum_{f_1,f_2> V}\Bigg)\frac{(f_1^2,f_2^2)\prod_{\ell \mid f_1f_2}\cpf}{f_1f_2\varphi(f_1^2)\varphi(f_2^2)}\sum_{n_1,n_2\geq 1}\frac{1}{\kappa_{2p}(n_1n_2)\varphi(n_1n_2)}\nonumber\\
+&Y\frac{(\log \log y^u)^2}{y^u}+\frac{Y}{\sqrt{y}}\Bigg).\label{newestsumthree}
\end{align}
Following as above, we have that there exists an $\epsilon>0$ such that 
\begin{equation}
\Bigg(\sum_{\substack{f_1\leq V\\f_2> V}}+\sum_{\substack{f_1> V\\f_2\leq V}}+\sum_{f_1,f_2> V}\Bigg)\frac{(f_1^2,f_2^2)\prod_{\ell \mid f_1f_2}\cpf}{f_1f_2\varphi(f_1^2)\varphi(f_2^2)}\sum_{n_1,n_2\geq 1}\frac{1}{\kappa_{2p}(n_1n_2)\varphi(n_1n_2)}\ll\frac{(\log \log V)}{V^{1-\epsilon}}.\label{pplusboundthree}
\end{equation}

Combining $\eqref{newestsumthree}$ and $\eqref{pplusboundthree}$ with $\eqref{psumfive}$ gives
\begin{align*}
S_1=& Y\sumfs\frac{\prod_{\ell \mid f_1f_2}\cpf}{f_1f_2}
\sum_{n_1,n_2\geq 1}\frac{2C_{p,f}(n_1,n_2)}{\varphi(L)n_1n_2}\prod_{\ell \mid (f_1f_2, n_1n_2)}\cpf^{-1}\nonumber\\
&+O_\alpha\Bigg(\frac{Y(\log \log 4p)^2}{(\log 4p)^{\alpha-3}}+\frac{Yu(\log p)^3(\log y)^2}{e^u} +\frac{\sqrt{p}(\log p)u^2(\log y)^2}{V^{\frac{1}{4}}}+\frac{Y}{\sqrt{y}}+\frac{Y(\log \log V)}{V^{1-\epsilon}}\\
&+y^{2u}(\log V)^2(u^2(\log y)^2+\log p)+\sum_{\substack{f_1,f_2\leq V}}f_1f_2\sum_{\substack{n_1,n_2 \leq y^u}}E(X,Y;L)+ \frac{Y(\log \log y^u)^2}{y^u}\Bigg).
\end{align*}
Choosing $$V:=(\log 4p)^{4\nu+4\gamma+12}\quad \text{and} \quad u:=\frac{\log 4p}{(\log \log 4p)^2},$$ yields
\begin{align}
S_1=&Y\sumfs\frac{\prod_{\ell \mid f_1f_2}\cpf}{f_1f_2}
\sum_{n_1,n_2\geq 1}\frac{2C_{p,f}(n_1,n_2)}{\varphi(L)n_1n_2}\prod_{\ell \mid (f_1f_2, n_1n_2)}\cpf^{-1}\nonumber\\
&+O_\alpha\left(\frac{Y}{(\log p)^\gamma}+\sum_{\substack{f_1,f_2\leq (\log 4p)^{4\nu+4\gamma+12}}}f_1f_2\sum_{\substack{n_1,n_2 \leq p^\epsilon}}E(X,Y;L) \right).\label{sonegooderrors}
\end{align}

We now show that the sums in the main term of $\eqref{sonegooderrors}$ converge. Recall that $L=4[n_1f_1^2,n_2f_2^2]$. From the properties of the Euler $\varphi$-function we have that
$$\frac{1}{f_1f_2}\frac{1}{n_1n_2\varphi(L)}=\frac{1}{f_1^2\varphi(f_1)f_2^2\varphi(f_2)}\frac{\varphi((n_1,f_1))\varphi((n_2,f_2))
\varphi((4n_1f_1^2,4n_2f_2^2))}{n_1\varphi(4n_1)(n_1,f_1)n_2\varphi(4n_2)(n_2,f_2)}.$$
To simplify our notation, we define the following multiplicative functions, 
$$g_1(f_i):= \frac{1}{f_i^2\varphi(f_i)}\quad {\rm and}\quad g_2(n_i,f_i):= 
\frac{\varphi((n_i,f_i))}{n_i\varphi(4n_i)(n_i,f_i)}\quad {\rm for}\:i=1,2.$$ Then we rewrite the main term of $\eqref{sonegooderrors}$ as $R(p)Y$ where
\begin{multline}
R(p):=\sumfs g_1(f_1)g_1(f_2)\prod_{\ell \mid f_1f_2} \cpf\sum_{\substack{n_1, n_2\geq 1}}2g_2(n_1,f_1)g_2(n_2,f_2)\\
 \times\varphi((4n_1f_1^2,4n_2f_2^2))C_{p,f}(n_1,n_2)\prod_{\ell \mid (f_1f_2,n_1n_2)}\cpf^{-1}. \label{usefulkp}
\end{multline}

Let $n_1:=n_1' n_1'', n_2 :=n_2' n_2''$. Then we say that the function $h(n_1,n_2)$ is \textit{multiplicative} if $h(n_1,n_2)=h(n_1',n_2')h(n_1'',n_2'')$ when $(n_1' n_2',n_1'' n_2'')=1$. From Lemma \ref{lemdseleven} we have that the functions of $n_1$ and $n_2$ are multiplicative, so the sum over $n_1,n_2$ in $\eqref{usefulkp}$ becomes
 \begin{align}
 \varphi((f_1^2,f_2^2))&\left\{2\sum_{\alpha_1,\alpha_2\geq 0}g_2(2^{\alpha_1},f_1)g_2(2^{\alpha_2},f_2)\varphi((2^{2+\alpha_1},2^{2+\alpha_2}))C_{p,f}(2^{\alpha_1},2^{\alpha_2})\right\}\nonumber\\
 \times&\left\{4\sum_{\alpha_1,\alpha_2\geq 0}g_2(p^{\alpha_1},f_1)g_2(p^{\alpha_2},f_2)\varphi((p^{\alpha_1},p^{\alpha_2}))C_{p,f}(p^{\alpha_1},p^{\alpha_2})\right\}\nonumber\\
 \times\prod_{\ell\nmid 2pf_1f_2}&\left\{4\sum_{\alpha_1,\alpha_2\geq 0}g_2(\ell^{\alpha_1},f_1)g_2(\ell^{\alpha_2},f_2)\varphi((\ell^{\alpha_1},\ell^{\alpha_2}))C_{p,f}(\ell^{\alpha_1},\ell^{\alpha_2})\right\}\nonumber\\
\times \prod_{\substack{\ell\mid f_1f_2 \\ \ell \neq p}}&\Bigg\{1+4\sum_{\substack{\alpha_1,\alpha_2\geq 0\\ \alpha_1+\alpha_2>0}}g_2(\ell^{\alpha_1},f_1)g_2(\ell^{\alpha_2},f_2)\frac{\varphi((\ell^{\alpha_1}f_1^2,\ell^{\alpha_2}f_2^2))C_{p,f}(\ell^{\alpha_1},\ell^{\alpha_2})}{\varphi((f_1^2,f_2^2))\cpf}\Bigg\},\label{bigkzeropone}
 \end{align}
by factoring out $\varphi((f_1^2,f_2^2))$ from each product in $\eqref{bigkzeropone}$.

We now invoke Lemma \ref{lemdseleven} and consider each term in $\eqref{bigkzeropone}$ separately. For the first term we have that $\displaystyle{g_2(2^{\alpha_i},f_i)=\frac{1}{2^{2\alpha_i+1}}}$ for $\alpha_i\geq 0$ and thus the first term in \eqref{bigkzeropone} becomes
\begin{equation}\label{kpelltwo}
1+2\sum_{\alpha\geq 1}\left[\frac{-1}{2}\right]^\alpha+\sum_{\alpha_1\geq 1}\left[\frac{-1}{2}\right]^{\alpha_1}\sum_{\alpha_2\geq 1}\left[\frac{-1}{2}\right]^{\alpha_2}=\frac{4}{9}.
\end{equation}
For the second term we have that $$g_2(p^{\alpha_i},f_i)=\begin{cases}\frac{1}{2} &\:{\rm if}\: \alpha_i=0, \\ \frac{p}{2p^{2\alpha_i}(p-1)} & \:{\rm if}\: \alpha_i> 0,\end{cases}$$
and thus the second sum in \eqref{bigkzeropone} becomes
\begin{align}
&1+\frac{2(p-2)}{p-1}\sum_{\alpha\geq 1}\left[\frac{1}{p}\right]^\alpha+\frac{p-2}{p-1}\sum_{\alpha_1\geq 1}\left[\frac{1}{p}\right]^{\alpha_1}\sum_{\alpha_2\geq 1}\left[\frac{1}{p}\right]^{\alpha_2}=1+\frac{(p-2)(2p-1)}{(p-1)^3}.\label{kpellp}
\end{align}
For the third sum, if $\ell\nmid 2pf_1f_2$ then $$g_2(\ell^{\alpha_i},f_i)=\begin{cases}\frac{1}{2} &\:{\rm if}\: \alpha_i=0, \\ \frac{\ell}{2\ell^{2\alpha_i}(\ell-1)} & \:{\rm if}\: \alpha_i> 0,\end{cases}$$
and thus the third sum in \eqref{bigkzeropone} becomes
\begin{align}
&1+\frac{2}{\ell-1}\sum_{\alpha\geq 1}\left[\frac{1}{\ell}\right]^\alpha\left[-1-\left(\frac{p-1}{\ell}\right)^2+\left[\ell-\left(\frac{p}{\ell}\right)\right]\left[\frac{1+(-1)^{\alpha}}{2}\right]\right]\nonumber \\
+&\frac{1}{\ell-1}\sum_{\alpha_1\geq 1}\left[\frac{1}{\ell}\right]^{\alpha_1}\sum_{\alpha_2\geq 1}\left[\frac{1}{\ell}\right]^{\alpha_2}\left[-1-\left(\frac{p-1}{\ell}\right)^2+\left[\ell-\left(\frac{p}{\ell}\right)\right]\left[\frac{1+(-1)^{\alpha_1+\alpha_2}}{2}\right]\right]\nonumber\\
=&1-\frac{(2\ell^3+3\ell^2-1)\left(\frac{p-1}{\ell}\right)^2+(3\ell^2-1)\left(\frac{p}{\ell}\right)-\ell^3+3\ell^2+\ell-1}{(\ell-1)(\ell^2-1)^2}.\label{kpellnmidf}
\end{align}
For the fourth term, if $\ell\mid f_1f_2$ then
 $$g_2(\ell^{\alpha_i},f_i)=\begin{cases}\frac{1}{2} &\:{\rm if}\: \alpha_i=0, \\ \frac{1}{2\ell^{2\alpha_i}} & \:{\rm if}\: \alpha_i> 0,\end{cases}$$ and thus the last term in \eqref{bigkzeropone} becomes
\begin{align}
&1+\sum_{\substack{\alpha_1,\alpha_2\geq 0 \\ \alpha_1+\alpha_2>0}}\left[\frac{1}{\ell^2}\right]^{\alpha_1+\alpha_2}\frac{\varphi((\ell^{\alpha_1} f_1^2,\ell^{\alpha_2}f_2^2))C_{p,f}(\ell^{\alpha_1},\ell^{\alpha_2})}{\varphi((f_1^2,f_2^2))\cpf}.\label{kpellmidfg}
\end{align}

From Lemma \ref{lemdseleven}, we must consider two cases. The first case is if $\min\{\alpha_1,\alpha_2\}=0$, the second case is if $\min\{\alpha_1,\alpha_2\}>0$. In the first case, if $\nu_\ell(f_1^2)>\nu_\ell(f_2^2)$ then we have that  $(\ell^{\alpha_1}f_1^2,f_2^2)=(f_1^2,f_2^2)$ and similarly if $\nu_\ell(f_2^2)> \nu_\ell(f_1^2)$ then  $(f_1^2,\ell^{\alpha_2}f_2^2)=(f_1^2,f_2^2).$
Thus, in this case $\eqref{kpellmidfg}$ becomes $$1+\frac{\ell-1}{2\ell}\sum_{\alpha\geq 1}\left[\frac{1}{\ell}\right]^\alpha\left(1+(-1)^\alpha\right)=1+\frac{1}{\ell(\ell+1)}.$$
If $\nu_\ell(f_1^2)=\nu_\ell(f_2^2)$, since $\ell\mid f_1f_2$, when $\alpha_1,\alpha_2\geq 1$ we have that $\varphi((\ell^{\alpha_1} f_1^2,\ell^{\alpha_2}f_2^2))=\ell^{\min\{\alpha_1,\alpha_2\}}\varphi((f_1^2,f_2^2))$. Hence, in this case $\eqref{kpellmidfg}$ becomes
\begin{align*}
&1+\frac{\ell-1}{\ell}\sum_{\alpha\geq 1}\left[\frac{1}{\ell}\right]^\alpha\left(1+(-1)^\alpha\right)+\frac{\ell-1}{2\ell}\sum_{\alpha_1,\alpha_2 \geq 1}\left[\frac{1}{\ell}\right]^{\alpha_1+\alpha_2}\left(1+(-1)^{\alpha_1+\alpha_2}\right) \nonumber \\
=&1+\frac{3\ell^2-1}{\ell(\ell-1)(\ell+1)^2}.
\end{align*}

By combining $\eqref{kpelltwo},\eqref{kpellp},\eqref{kpellnmidf}$ and $\eqref{kpellmidfg}$ we have that the sum over $n_1,n_2$ in $\eqref{usefulkp}$ becomes
\begin{equation}\label{thensumforkp}
\frac{4}{9}\varphi((f_1^2,f_2^2))F_0(p)\prod_{\ell \nmid 2pf_1f_2}F_1(\ell)\prod_{\substack{\ell \mid f_1f_2\\ \ell \neq p}}F_2(\ell,f_1,f_2),
\end{equation}
where, for any odd prime $\ell$, we make the definitions
\begin{align*}
F_0(p)&:=1+\frac{(p-2)(2p-1)}{(p-1)^3},\\
F_1(\ell)&:= 1-\frac{(2\ell^3+3\ell^2-1)\left(\frac{p-1}{\ell}\right)^2+(3\ell^2-1)\left(\frac{p}{\ell}\right)-\ell^3+3\ell^2+\ell-1}{(\ell-1)(\ell^2-1)^2},\\
F_2(\ell,f_1,f_2)&:=\begin{cases} 
1+\frac{1}{\ell(\ell+1)} & \text{ if }\nu_\ell(f_1^2)\neq \nu_\ell(f_2^2), \\ 
1+\frac{3\ell^2-1}{\ell(\ell-1)(\ell+1)^2}& \text{ if }\nu_\ell(f_1^2)= \nu_\ell(f_2^2). 
\end{cases}
\end{align*}
Hence, $\eqref{usefulkp}$ becomes
\begin{equation}
R(p)=\frac{4}{9}F_0(p)\prod_{\ell \neq 2,p}F_1(\ell)\sumfs g_1(f_1)g_1(f_2)\varphi((f_1^2,f_2^2))\prod_{\ell \mid f_1f_2} \cpf\frac{F_2(\ell,f_1,f_2)}{F_1(\ell)}. \label{kzeropnew}
\end{equation}

We have that $$g_1(\ell^{\alpha_i})=\begin{cases}1 &\:{\rm if}\: \alpha_i=0, \\ \frac{\ell}{\ell^{3\alpha_i}(\ell-1)} & \:{\rm if}\: \alpha_i> 0,\end{cases}$$ so the sum in $\eqref{kzeropnew}$ over $f_1,f_2$ may be factored as
\begin{align}
\prod_{\ell\neq 2,p}\Bigg\{1+&\frac{\ell\left[1+\left(\frac{p(p-1)^2}{\ell}\right)\right]}{(\ell-1)F_1(\ell)}
\Bigg[\sum_{\alpha\geq 1}\left[\frac{1}{\ell^3}\right]^\alpha \left(F_2(\ell,\ell^\alpha,1)+F_2(\ell,1,\ell^\alpha)\right)\nonumber\\
+&\sum_{\alpha_1\geq 1}\left[\frac{1}{\ell^3}\right]^{\alpha_1}\sum_{\alpha_2\geq 1}\left[\frac{1}{\ell^3}\right]^{\alpha_2}\ell^{2\min\{\alpha_1,\alpha_2\}}F_2(\ell,\ell^{\alpha_1}, \ell^{\alpha_2})\Bigg]\Bigg\}\nonumber\\
=\prod_{\ell\neq 2,p}\Bigg\{1+&\frac{\left[1+\left(\frac{p(p-1)^2}{\ell}\right)\right]}{F_1(\ell)}
\Bigg[\frac{2\ell^3+3\ell^2-\ell-1}{(\ell^2-1)^3}\Bigg]\Bigg\}.\label{kzeropnewtwo}
\end{align}
Combining $\eqref{kpellnmidf}, \eqref{kzeropnew}$ and $\eqref{kzeropnewtwo}$ gives the equation 
\begin{multline}
R(p)=\frac{4}{9}F_0(p)\prod_{\ell\neq 2,p}\Bigg(1-\frac{(2\ell^4+3\ell^3)\left(\frac{p-1}{\ell}\right)^2+\ell^3\left(\frac{p}{\ell}\right)-\ell^4+2\ell^3+4\ell^2-1}{(\ell^2-1)^3}\\
+\frac{(2\ell^3+3\ell^2-\ell-1)\left(\left(\frac{p-1}{\ell}\right)^2+\left(\frac{p}{\ell}\right)-1-\left(\frac{p(p-1)^2}{\ell}\right)\right)}{(\ell^2-1)^3}\Bigg).\label{ctwopcont}
\end{multline}
Note that since $\ell\neq p$ we have that $$\left(\frac{p-1}{\ell}\right)^2+\left(\frac{p}{\ell}\right)-1-\left(\frac{p(p-1)^2}{\ell}\right)=0,$$ and hence $\eqref{ctwopcont}$ becomes
\begin{equation*}
R(p)=\frac{4}{9}F_0(p)\prod_{\ell\neq 2,p}\Bigg(1-\frac{(2\ell^4+3\ell^3)\left(\frac{p-1}{\ell}\right)^2+\ell^3\left(\frac{p}{\ell}\right)-\ell^4+2\ell^3+4\ell^2-1}{(\ell^2-1)^3}\Bigg).
\end{equation*}
Since $$F_0(p)\left(1-\frac{2p^4+3p^3-p^4+2p^3+4p^2-1}{(p^2-1)}\right)^{-1}=1+O\left(\frac{1}{p^2}\right),$$ we have that $R(p)=C_2(p)(1+O(p^{-2}))$ and the result follows.
\end{proof}

\section{Proof of Lemma \ref{lemdseleven}}
We now provide the proof of Lemma \ref{lemdseleven} stated in the previous section. The function $C_{p,f}(n_1,n_2)$ is very similar to the function $C_{N,f}(n)$ considered in  \cite[Lemma 11]{CDES:1} and we will make frequent reference to their paper in the following proof. For the duration of this section let $a_1,a_2,f_1,f_2,n_1,n_2$ be integers such that $f_1$ and $f_2$ are odd and recall that $L=4[n_1f_1^2,n_2f_2^2]$.

\begin{proof}(Proof of Lemma \ref{lemdseleven})
We first show that $C_{p,f}(n_1,n_2)$ is multiplicative in two variables. Let $n_1:=n_1' n_1'', n_2 :=n_2' n_2''$. Recall the function $C_{p,f}(n_1,n_2)$ is multiplicative if $C_{p,f}(n_1,n_2)=C_{p,f}(n_1',n_2')C_{p,f}(n_1'',n_2'')$ when $(n_1' n_2',n_1'' n_2'')=1$ and $C_{p,f}(1,1)=1$.

Define $(a',f,n')$ and $(a'',f,n'')$ analogously to $\eqref{afn}$. Now, suppose that $(n_1' n_2',n_1'' n_2'')=1$, then we can assume that at least one of $n_1' n_2'$ and $ n_1'' n_2'' $ is odd. Without loss of generality, we assume that $n_1' n_2'$ is odd and hence $n'$ is also odd. Therefore there exist integers $h_1',h_1'', h_2',h_2'',h',h''$ such that
\begin{equation} \label{cpmult}
4n_1'' h_1''+n_1' h_1'= 4n_2'' h_2''+n_2' h_2'=4n'' h''+n' h'=1.
\end{equation}
Since we assumed that $n_1'n_2'$ is odd, we have that $S^{(2)}(a',n')=1$ and by definition,
\begin{align}
&C_{p,f}(n_1',n_2')C_{p,f}(n_1'',n_2'')\nonumber\\
=&\sum_{a_1'\in(\ZZ/n_1'\ZZ)^*}\left(\frac{a_1'}{n_1'}\right)\sum_{\substack{a_2'\in(\ZZ/n_2'\ZZ)^* \\  a_1'f_1^2\equiv a_2'f_2^2\imod{(4n_1'f_1^2,4n_2'f_2^2)}}}\left(\frac{a_2'}{n_2'}\right)\prod_{\substack{\ell\mid n_1'n_2' \\ \ell \neq 2}}C_p^{(\ell)}(a',f,n')\nonumber\\
\times&\sum_{\substack{a_1''\in(\ZZ/4n_1''\ZZ)^* \\ a_1''\equiv 1\imod{4}}}\left(\frac{a_1''}{n_1''}\right)\sum_{\substack{a_2''\in(\ZZ/4n_2''\ZZ)^* \\ a_2''\equiv 1\imod{4} \\  a_1''f_1^2\equiv a_2''f_2^2\imod{(4n_1''f_1^2,4n_2''f_2^2)}}}\left(\frac{a_2''}{n_2''}\right)S^{(2)}(a'',n'')
\prod_{ \ell \mid n_1''n_2'' }C_p^{(\ell)}(a'',f,n'').\label{cpfeqone}
\end{align}
From $\eqref{cpmult}$ we have that $a_1'=a_1'(4n_1'' h_1''+n_1' h_1')$, and hence
$$\left(\frac{a_1'}{n_1'}\right)=\left(\frac{a_1'(4n_1'' h_1''+n_1' h_1')}{n_1'}\right)=\left(\frac{a_1'4n_1'' h_1''}{n_1'}\right)=\left(\frac{a_1'4n_1'' h_1''+a_1''n_1' h_1'}{n_1'}\right).$$
Similarly, we have that
$$\left(\frac{a_1''}{n_1''}\right)=\left(\frac{a_1'4n_1'' h_1''+a_1''n_1' h_1'}{n_1''}\right),$$
 and the analogous results for $a_2'$ and $a_2''$. 
 
We assumed that $(n',n'')=1$ and we have that $n_1'\equiv n_2'\equiv 0 \imod{(4n_1'f_1^2,4n_2'f_2^2)}.$ Thus, 
 \begin{align*}
 a_1'f_1^2&=a_1'(4n_1'' h_1''+n_1' h_1')f_1^2\equiv (a_1'4n_1'' h_1'')f_1^2\equiv(a_1'4n_1'' h_1''+a_1''n_1'h_1')f_1^2\\
 &\equiv(a_2'4n_2'' h_2''+a_2''n_2'h_2')f_2^2\imod{(4n_1'n_1''f_1^2,4n_2'n_2''f_2^2)}.
 \end{align*}
Hence, the right hand side of $\eqref{cpfeqone}$ becomes
\begin{align*}
&\sum_{\substack{a_1'\in(\ZZ/n_1'\ZZ)^* \\ a_1''\in(\ZZ/4n_1''\ZZ)^* \\ a_1''\equiv 1\imod{4}}}\left(\frac{a_1'4n_1''h_1''+a_1''n_1'h_1'}{n_1' n_1''}\right)
\sum_{\substack{a_2'\in(\ZZ/n_2'\ZZ)^*,  a_2''\in(\ZZ/4n_2''\ZZ)^* \\ a_2''\equiv 1\imod{4} \\  (a_1'4n_1'' h_1''+a_1''n_1'h_1')f_1^2\equiv(a_2'4n_2'' h_2''+a_2''n_2'h_2')f_2^2\imod{(4n_1'n_1''f_1^2,4n_2'n_2''f_2^2)}}}\\
\times&\left(\frac{a_2'4n_2''h_2''+a_2''n_2'h_2'}{n_2' n_2''}\right)S^{(2)}(a'4n''h''+a''n'h',n' n'')\prod_{\ell \mid n_1'n_1''n_2'n_2''}C_p^{(\ell)}(a'4n''h''+a''n'h',f,n' n'')\\
=&C_{p,f}(n_1'n_1'',n_2'n_2''),
\end{align*}
and thus $C_{p,f}(n_1,n_2)$ is multiplicative.

Now let $\alpha_i:=\nu_\ell(n_i)$ for $i=1,2$ and consider the sum $C_{p,f}(\ell^{\alpha_1},\ell^{\alpha_2})$. 
Without loss of generality it suffices to consider the case when $(a,f,n)=(a_1,f_1,\ell^{\alpha_1})$ since $C_{p,f}(\ell^{\alpha_1},\ell^{\alpha_2})$ is symmetric.

We first consider $C_{p,f}(2^{\alpha_1},2^{\alpha_2})$, when $\alpha_1\geq 1$ and $\alpha_1\geq \alpha_2$. From Lemma \ref{lemdsten} we have that 
\begin{align*}
C_{p,f}(2^{\alpha_1},2^{\alpha_2})&=2\sum_{\substack{a_1\in(\ZZ/2^{\alpha_1+2}\ZZ)^* \\ a_1\equiv 5\imod{8}}}\left(\frac{a_1}{2^{\alpha_1}}\right)
\sum_{\substack{a_2\in(\ZZ/2^{\alpha_2+2}\ZZ)^* \\ a_2\equiv 1\imod{4} \\ a_1f_1^2\equiv a_2f_2^2 \imod{2^{2+\alpha_2}(f_1^2,f_2^2)}}}\left(\frac{a_2}{2^{\alpha_2}}\right)\\
&=2\sum_{\substack{a_1\in(\ZZ/2^{\alpha_1+2}\ZZ)^* \\ a_1\equiv 5\imod{8}}}\left(\frac{a_1}{2}\right)^{\alpha_1+\alpha_2}=2^{\alpha_1}\sum_{\substack{a_1\in(\ZZ/8\ZZ)^* \\ a_1\equiv 5\imod{8}}}\left(\frac{a_1}{2}\right)^{\alpha_1+\alpha_2}=2^{\alpha_1}(-1)^{\alpha_1+\alpha_2}.
\end{align*}

Next we consider when $\ell\nmid 2f_1f_2$. In this case, 
\begin{align*}
C_{p,f}(\ell^{\alpha_1},\ell^{\alpha_2})=&\sum_{\substack{a_1\in(\ZZ/4\ell^{\alpha_1}\ZZ)^* \\ a_1\equiv 1\imod{4}}}\left(\frac{a_1}{\ell^{\alpha_1}}\right)
C_p^{(\ell)}(a_1,f_1,\ell^{\alpha_1})\sum_{\substack{a_2\in(\ZZ/4\ell^{\alpha_2}\ZZ)^* \\ a_2\equiv 1\imod{4}\\  a_1f_1^2\equiv a_2f_2^2
\imod{4(f_1^2,f_2^2)\ell^{\alpha_2}}}}\left(\frac{a_2}{\ell^{\alpha_2}}\right).
\end{align*}
Since $\ell \nmid f_1f_2$, we have that the inner sum above becomes
$$\sum_{\substack{a_2\in(\ZZ/4\ell^{\alpha_2}\ZZ)^* \\ a_2\equiv 1\imod{4}\\  a_1f_1^2\equiv a_2f_2^2\imod{\ell^{\alpha_2}}}}\left(\frac{a_2}{\ell^{\alpha_2}}\right)=\left(\frac{a_1}{\ell}\right)^{\alpha_2},$$ and hence
\begin{equation}
C_{p,f}(\ell^{\alpha_1},\ell^{\alpha_2})=\sum_{\substack{a_1\in(\ZZ/4\ell^{\alpha_1}\ZZ)^* \\ a_1\equiv 1\imod{4}}}\left(\frac{a_1}{\ell}\right)^{\alpha_1+\alpha_2}C_p^{(\ell)}(a_1,f_1,\ell^{\alpha_1}).\label{cpflone}
\end{equation}

We have that the sum in $\eqref{cpflone}$ differs only by the exponent on the Legendre symbol from the quantity $C_{N,f}(\ell^\alpha)$ considered in \cite[Equation (38)]{CDES:1}. It follows that their method can be applied analogously to obtain the formula for $C_{p,f}(\ell^{\alpha_1},\ell^{\alpha_2})$ in the case $\ell \nmid 2f_1f_2$ given in the statement of the lemma.

It remains to consider the case when $\ell \mid f_1f_2$. We assumed that $(a,f,n)=(a_1,f_1,\ell^{\alpha_1})$, and therefore $(4\ell^{\alpha_1}f_1^2,4\ell^{\alpha_2}f_2^2)=4\ell^{\alpha_2+\nu_\ell(f_2^2)}(f_1^2/\ell^{\nu_\ell(f_1^2)},f_2^2/\ell^{\nu_\ell(f_2^2)}) $ and in this case
\begin{align}
C_{p,f}(\ell^{\alpha_1},\ell^{\alpha_2})&=\cpf\sum_{\substack{a_1\in(\ZZ/4\ell^{\alpha_1}\ZZ)^* \\ a_1\equiv 1\imod{4}}}\left(\frac{a_1}{\ell^{\alpha_1}}\right)
\sum_{\substack{a_2\in(\ZZ/4\ell^{\alpha_2}\ZZ)^* \\ a_2\equiv 1\imod{4}\\  a_1f_1^2\equiv a_2f_2^2
\imod{4\ell^{\alpha_2+\nu_\ell(f_2^2)}(f_1^2/\ell^{\nu_\ell(f_1^2)},f_2^2/\ell^{\nu_\ell(f_2^2)})}}}\left(\frac{a_2}{\ell^{\alpha_2}}\right)\nonumber\\
&=\cpf\sum_{\substack{a_1\in(\ZZ/\ell^{\alpha_1}\ZZ)^* }}\left(\frac{a_1}{\ell^{\alpha_1}}\right)
\sum_{\substack{a_2\in(\ZZ/\ell^{\alpha_2}\ZZ)^* \\\  a_1f_1^2\equiv a_2f_2^2
\imod{\ell^{\alpha_2+\nu_\ell(f_2^2)}}}}\left(\frac{a_2}{\ell^{\alpha_2}}\right).\label{ellmidfprop}
\end{align}

If $\alpha_2=0$,  since $a_2f_2^2\equiv 0\imod{\ell^{\nu_\ell(f_2^2)}}$, for a solution to exist we must have that $a_1f_1^2\equiv 0 \imod{\ell^{\nu_\ell(f_2^2)}}$ as well. Thus, we require that  $\nu_\ell(f_2^2)\leq \nu_\ell(f_1^2)$. This gives 
\begin{align*}
C_{p,f}(\ell^{\alpha_1},1)&=\cpf\sum_{\substack{a_1\in(\ZZ/\ell^{\alpha_1}\ZZ)^* }}\left(\frac{a_1}{\ell^{\alpha_1}}\right)=\cpf \ell^{\alpha_1-1}\sum_{a_1\in(\ZZ/\ell\ZZ)^*}\left(\frac{a_1}{\ell}\right)^{\alpha_1}\\
&=\cpf\ell^{\alpha_1-1}\begin{cases} \ell-1 &\text{ if }2\mid \alpha_1 \text{ and }\nu_\ell(f_2^2)\leq \nu_\ell(f_1^2), \\ 0 & \text{ otherwise}.\end{cases}
\end{align*}

If $\alpha_2\geq 1$ then $a_1f_1^2\equiv a_2f_2^2
\imod{\ell^{\alpha_2+\nu_\ell(f_2^2)}}$ has a solution if and only if $(a_2f_2^2,\ell^{\alpha_2+\nu_\ell(f_2^2)})=\ell^{\nu_\ell(f_2^2)}\mid a_1f_1^2$. Thus, we require that
$\nu_\ell(f_2^2)\leq \nu_\ell(f_1^2)$. Now define integers $h_1$ and $h_2$ such that $f_1^2=h_1^2\ell^{\nu_\ell(f_1^2)}, f_2^2=h_2^2\ell^{\nu_\ell(f_2^2)}$ and $(h_1h_2,\ell)=1$. Then the sum over $a_2$ in $\eqref{ellmidfprop}$ becomes
\begin{align*}
\sum_{\substack{a_2\in(\ZZ/\ell^{\alpha_2}\ZZ)^* \\ a_1f_1^2\equiv a_2f_2^2
\imod{\ell^{\alpha_2+\nu_\ell(f_2^2)}}}}\left(\frac{a_2}{\ell}\right)^{\alpha_2} &= \sum_{\substack{a_2\in(\ZZ/\ell^{\alpha_2}\ZZ)^* \\ 
a_2\equiv a_1h_1^2(h_2^2)^{-1}\ell^{\nu_\ell(f_1^2)-\nu_\ell(f_2^2)}
\imod{\ell^{\alpha_2}}}}\left(\frac{a_2}{\ell}\right)^{\alpha_2}\\
&=\left(\frac{a_1\ell^{\nu_\ell(f_1^2)-\nu_\ell(f_2^2)}}{\ell}\right)^{\alpha_2}=\left(\frac{\ell^{(\nu_\ell(f_1^2)-\nu_\ell(f_2^2))}}{\ell}\right)
\left(\frac{a_1}{\ell}\right)^{\alpha_2}.
\end{align*}
Then in this case we have that
\begin{align*}
C_{p,f}(\ell^{\alpha_1},\ell^{\alpha_2})&=
\left(\frac{\ell^{(\nu_\ell(f_1^2)-\nu_\ell(f_2^2))}}{\ell}\right)\cpf\sum_{a_1\in(\ZZ/\ell^{\alpha_1}\ZZ)^*}
\left(\frac{a_1}{\ell}\right)^{\alpha_1+\alpha_2}\\
&=\cpf\ell^{\alpha_1-1}\begin{cases} \ell-1 &\text{ if }2\mid \alpha_1+\alpha_2 \text{ and } \nu_\ell(f_1^2)=\nu_\ell(f_2^2), \\ 0 & \text{ otherwise}.\end{cases}
\end{align*}

Since $C_{p,f}(n_1,n_2)$ is multiplicative, we have for a prime $\ell$ that
$$C_{p,f}(n_1,n_2)=\prod_{\ell \mid n_1n_2}C_{p,f}(\ell^{\nu_\ell(n_1)},\ell^{\nu_\ell(n_2)}).$$
Then it follows from the above results that
\begin{align*}
C_{p,f}(\ell^{\nu_\ell(n_1)},\ell^{\nu_\ell(n_2)})\ll \left(\frac{\ell^{\nu_\ell(n_1)+\nu_\ell(n_2)}}{(\ell^{\nu_\ell(n_1)},\ell^{\nu_\ell(n_2)})\kappa_{2p}
(\ell^{\nu_\ell(n_1)+\nu_\ell(n_2)})}\right)
\prod_{\ell \mid (f_1f_2, n_1n_2)}\cpf
 \end{align*}
 with an absolute constant. The result follows by taking the product over $\ell \mid n_1n_2$. 
\end{proof}

\appendix
\section{by Sumit Giri}
The goal of this appendix is to determine the average of the function $C_2(p)$, defined in $\eqref{actualkzerop}$, over the set of primes up to $X$. We state this result as follows.
\begin{thrm}\label{thrm: Giri}
	Let $C_2(p)$ be the constant defined in $\eqref{actualkzerop}$. Then for any $M>0$, we have  
	$$\sum_{p\leq X}C_2(p)=C \pi(X)\left(1+O\left(\frac{1}{(\log X)^M}\right)\right),$$
	where $C$ is given by
		\begin{align*}
		C:=&\prod_{\ell}\left(1-\frac{(2\ell^4+3\ell^3)(\ell-2)-(\ell-1)(\ell^4-2\ell^3-4\ell^2+1)}{(\ell-1)(\ell^2-1)^3}\right).
		\end{align*}
\end{thrm}

\begin{proof}
First recall the function 	\begin{align*}
C_2(p):=&\frac{4}{9}\prod_{\ell>2}\Bigg(1-\frac{(2\ell^4+3\ell^3)\left(\frac{p-1}{\ell}\right)^2+\ell^3\left(\frac{p}{\ell}\right)-\ell^4+2\ell^3+4\ell^2-1}{(\ell^2-1)^3}\Bigg).
\end{align*}
	For each prime $p$, we define the function $f_p(\ell)$ by 
	$$f_p(\ell):=(2\ell^4+3\ell^3)\left(\frac{p-1}{\ell}\right)^2+\ell^3\left(\frac{p}{\ell}\right)-\ell^4+2\ell^3+4\ell^2-1.$$
	
	Then 
	\begin{align}
	\sum_{p\le X}C_2(p)&=\sum_{p\le X}\sum_{\substack{n\geq 1 \\n \text{ odd}}}(-1)^{\omega(n)}\frac{\mu^2(n)}{n^2}\prod_{\ell\mid n}\left(\frac{f_p(\ell)\ell^2}{(\ell^2-1)^3}\right)\label{eq_6}
	\end{align}
	where $\omega(n)$ is the number of distinct prime factors of $n$ and $\mu(n)$ is the M\"{o}bius function.
	In $\eqref{eq_6}$ we restrict the inner sum to integers $n\le (\log X)^M$, which gives rise to an error term of size  $ X/(\log X)^B$. Thus, the main term in (\ref{eq_6}) becomes
	\begin{align}
	\sum_{\substack{ 1\leq n\leq (\log X)^M \\ n \text{ odd}}}(-1)^{\omega(n)}\mu^2(n)\prod_{\ell\mid n}\frac{1}{(\ell^2-1)^3}\sum_{p\le X}\prod_{\ell\mid n}f_p(\ell)\label{eq_5}.
	\end{align}
	Now we define three multiplicative functions $a(\cdot)$, $b(\cdot)$, and $c(\cdot)$, supported on square-free integers by \begin{align}
	a(\ell)=2\ell^4+3\ell^3, \quad b(\ell)=\ell^3, \quad c(\ell)=2\ell^3-\ell^4+4\ell^2-1.\label{eq_7}
	\end{align}  
	In other words, $f_p(\ell)=a(\ell)\left(\frac{p-1}{\ell}\right)^2+b(\ell)\left(\frac{p}{\ell}\right)+c(\ell)$ if $\ell$ is prime.
	For every odd integer $n$, we have
	\begin{align}
	\sum_{p\le X}\prod_{\ell\mid n}f_p(\ell)&=\sum_{p\le  X}\sum_{\substack{n_1n_2n_3=n\\ (n_1,n_2)=(n_2,n_3)=(n_3,n_1)=1}}a(n_1)\left(\frac{p-1}{n_1}\right)^2b(n_2)\left(\frac{p}{n_2}\right)c(n_3)\nonumber\\
	&=\sum_{\substack{n_1n_2n_3=n\\ (n_1,n_2)=(n_2,n_3)=(n_3,n_1)=1}}a(n_1)b(n_2)c(n_3)\sum_{\underset{(p-1,n_1)=1}{p\le X}}\left(\frac{p}{n_2}\right)\nonumber\\
	&=\sum_{\substack{n_1n_2n_3=n\\ (n_1,n_2)=(n_2,n_3)=(n_3,n_1)=1}}a(n_1)b(n_2)c(n_3)\sum_{b\in (\Z/n_2\Z)^*}\left(\frac{b}{n_2}\right)\sum_{\substack{p\leq X \\ p\equiv b \imod{n_2}\\(p-1,n_1)=1}}1.\label{eq_8}
	\end{align}
	We now consider the multiplicative function $$\delta(m):=\#\{ 1\le a\le m-1: \, (a,m)=(a+1,m)=1\}.$$ Then, using the Siegel--Walfisz theorem for modulus  $n_1n_2\le (\log X)^M$, the last sum in $\eqref{eq_8}$ is equal to $\frac{\delta(n_1)\text{Li}(X)}{\phi(n_1n_2)}$ with an error term $E(X)$, bounded by $X\cdot\exp(-c(\log X)^{\frac{1}{2}})$ for some constant $c$. Note that since $n \leq (\log X)^M$, the error term  $E(X)$ does not affect the size of the error term in the statement of the theorem.\par
	
	Finally, the main term in (\ref{eq_8}) is then equal $$\text{Li}(X)\sum_{\substack{n_1n_2n_3=n \\ (n_1,n_2)=(n_2,n_3)=(n_3,n_1)=1}}\frac{\delta(n_1)a(n_1)b(n_2)c(n_3)}{\phi(n_1n_2)}\sum_{b\in (\Z/n_2\Z)^*}\left(\frac{b}{n_2}\right).$$
	We have that the last sum in the above expression is zero if $n_2\ge 3$. Also since $n$ is odd, the only contribution comes from the trivial case $n_2=1$.\par
	In that case, (\ref{eq_5}) becomes
	\begin{equation}
	\label{appendix-delta}
		\pi(X)\sum_{\substack{n\geq 1 \\ n \text{ odd}}}(-1)^{\omega(n)}\mu^2(n)\prod_{\ell\mid n}\frac{1}{(\ell^2-1)^3}\sum_{\substack{n_1n_3=n \\(n_3,n_1)=1}}\frac{\delta(n_1)a(n_1)c(n_3)}{\phi(n_1)}+O\left(\frac{X}{(\log X)^{M+1}}\right).
	\end{equation}
	We have that the function inside the outer sum in the main term in $\eqref{appendix-delta}$ is multiplicative in $n$, and hence can be written as the Euler product
	$$\prod_{\ell>2}\left(1-\frac{(a(\ell)\delta(\ell)/\phi(\ell))+c(\ell)}{(\ell^2-1)^3}\right).$$
	Then the result follows from the fact that $\delta(\ell)=(\ell-2)$. This completes the proof.
	
\end{proof}

\end{document}